\documentclass{amsart}
\usepackage[utf8]{inputenc}
\usepackage{amssymb}
\usepackage{amsthm}
\usepackage[colorlinks]{hyperref}
\usepackage{enumitem}
\usepackage[dvips]{graphicx}
\usepackage{transparent}
\usepackage{wrapfig}
\usepackage[usenames,dvipsnames,svgnames,table]{xcolor}
\definecolor{darkred}{rgb}{0.55, 0.0, 0.0}
\hypersetup{colorlinks,linkcolor=darkred}
\usepackage{enumitem}
\usepackage{marginnote}
\usepackage[capitalize,noabbrev,nameinlink]{cleveref}

\numberwithin{equation}{section}

\newtheorem{thm}{Theorem}[section]
\newtheorem{lemma}[thm]{Lemma}

\newtheorem{remark}[thm]{Remark}
\newtheorem{proposition}[thm]{Proposition}

\theoremstyle{definition}
\newtheorem{defn}[thm]{Definition}

\newtheorem{open}[thm]{Open problem}

\def\eqref#1{\textcolor{darkred}{(\ref{#1})}}
\DeclareMathOperator{\inte}{int}

\begin{document}
\title[Dynamics of the Takagi function and the shadowing property]{Dynamics of the Takagi function and the shadowing property}

\date{\today\ (\the\time)}

\author{Zoltán Buczolich}
\address{Department of Analysis, ELTE Eötvös Loránd University, Pázmány Péter Sétany 1/c, H-1117 Budapest, Hungary}
\email{zoltan.buczolich@ttk.elte.hu}

\author{Jes\'us Llorente}
\address{IMI, Departamento de An{\'a}lisis Matem{\'a}tico y Matem{\'a}tica Aplicada,
	Facultad Ciencias Matem{\'a}ticas, Universidad Complutense, 28040, Madrid, Spain}
\email{jesllore@ucm.es}

\thanks{This research was partially supported by  grant PID2022-138758NB-I00 from the Spanish Ministry of Science and Innovation.}
\begin{abstract}
The Takagi function $T:[0,1]\to \mathbb{R}$ is a classical example of a continuous nowhere differentiable function. In this paper, we study the discrete dynamical system generated by the Takagi function.  First, we prove that for almost every point $x\in [0,1]$, the orbit $(T^n(x))_n$ converges to $2/3$. We introduce the family of Takagi maps, given by $\textbf{T}_\gamma=\gamma \cdot T$, where $\gamma>0$ is a parameter. We also study the shadowing property for this family of maps. We show that the Takagi function has the shadowing property. Additionally, we provide two distinct techniques that allow us to find values of the parameter $\gamma$ for which $\textbf{T}_\gamma$ fails to have the shadowing property. Finally, we pose some open questions.
\end{abstract}
\maketitle

\section{Introduction}\label{IntroSection}
The Takagi function is perhaps the simplest example of a continuous nowhere differentiable function. Although the original paper was published in 1903 (see \cite{T}), it remained overlooked for several decades in the Western World. As a consequence, many authors rediscovered it, bringing to light its striking properties and its connections to various fields, including probability theory, number theory, and mathematical analysis.

Among the various ways of defining the Takagi function that can be found in the existing literature (see Section 4 of \cite{AK}, for instance), we adopt the following: let $D$ be the set of all dyadic numbers in the interval $[0,1]$ and we write 
$\displaystyle D=\cup_{n=1}^\infty D_n$, where 
$$
D_{n}=\left \{\frac{k}{2^{n-1}}:k=0,1\ldots, 2^{n-1}\right \}.
$$
We define the Takagi function $T:[0,1]\to\mathbb{R}$ as 
$$
T(x)=\sum_{n=1}^{\infty}g_n(x)=\lim_{m\to \infty}G_m(x),
$$
where $g_n(x)=\text{dist}\,(x, D_n)$ denotes the distance from $x$ to the set $D_{n}$ and $G_m=g_1+\cdots + g_m$. Observe that $G_m$ is a polygonal function whose nodes are located at the points of $D_{m+1}$. Furthermore, these functions converge uniformly to the Takagi function while approximating it monotonically from below (see Figure \ref{Gr} below). 
\begin{figure}[t]
	\begin{center}
		 \includegraphics[width=0.7\linewidth]{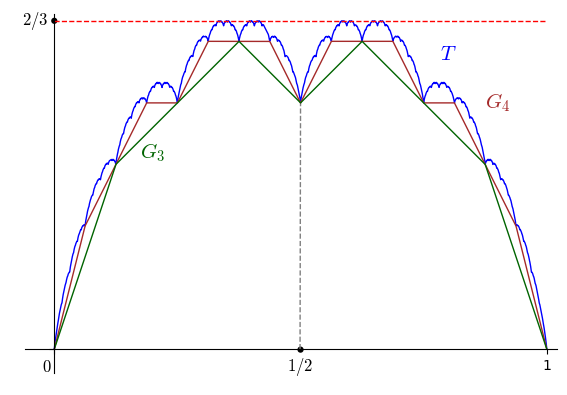}
		\caption{The graphs of $G_3$, $G_4$ and $T$.}
		\label{Gr}
	\end{center}
\end{figure}
Concerning the properties of its graph, it is immediate to see that it is symmetric about the line $x = 1/2$, and it has strict local minima at the dyadic points. In 1959, J. P. Kahane showed that the maximum value of the Takagi function is $2/3$ and determined the set of points where this maximum is attained (see \cite{Ka}). Moreover, the first author proved that for almost every ordinate in $ [0, \frac{2}{3}]$, the corresponding level set is finite (see \cite{Bu}). Although the Takagi function is not a self-similar function, it  satisfies the following ``self-affine" property (see Theorem 4.2 of \cite{L}): for every $n$ and $x\in D_n$  we have  
\begin{equation}\label{selfAffineProperty}
	T\left (x+\frac{y}{2^{n-1}} \right ) = T(x)+\frac{1}{2^{n-1}} \left [ T(y)+y\, G'^{+}_{n-1}(x) \right ]\quad \text{ for all }\, y\in [0,1].
\end{equation}

The aforementioned properties will be useful in obtaining the results in this paper. The surveys \cite{AK} and \cite{L}, as well as the thesis \cite{JLL}, contain a wealth of information about the Takagi function and related topics.

In this paper, we consider the discrete dynamical system generated by the Takagi function, namely the one-dimensional dynamical system given by 
\begin{equation}\label{dds}
x_{n+1}= T(x_n),\quad x_0\in [0,1].
\end{equation}
In the sequel, the notation $T^n$ represents the composition of $T$ with itself $n$ times. Concerning the approach to the Takagi function from dynamical systems, we must mention the work of Y. Yamaguchi, K. Tanikawa, and N. Mishima (see \cite{YT}), in which the Takagi function is expressed as the invariant curve of a certain discrete dynamical system. The study of the system \eqref{dds} provides a new perspective on the Takagi function through discrete dynamics. 

Section \ref{S1} focuses on describing the behavior of an orbit of the system \eqref{dds}. First, we compute the set of fixed points of the Takagi function.  We then prove that for almost every point $x\in [0,1]$ the orbit $(T^n(x))_n$ converges to $2/3$. To achieve this, we use the following result, which is an immediate consequence of Theorem 2.4. in \cite{BLL}:
\begin{proposition}\label{NAD1}
For every $n$, the function $T^n$ is nowhere approximately differentiable.
\end{proposition}

When a computer is used to draw an orbit $(T^n(x))_n$ for some $x\in [0,1]$,  we must keep in mind that the value of the Takagi function at a point, as calculated by the computer, will always be an approximation of the true value, since it can only sum a finite number of terms in the series defining the Takagi function. In other words, a computer is only capable of computing the value of the function $G_n$ for a specific value of $n$. Moreover, a computer also makes rounding errors. This leads to the notion of a pseudo-orbit: 

\begin{defn}
	Let $f:\mathbb{R}\to \mathbb{R}$ be a continuous function. For a given $\delta>0$, a sequence $(x_n)_n$ is said to be a $\delta$-\emph{pseudo orbit} of $f$ provided that 
	$$
	|f(x_n)-x_{n+1}|<\delta\quad \text{ for all }\,n\geq 0.
	$$
\end{defn}

This notion dates back at least to G. D. Birkhoff (see \cite{Bi}), and its study is closely related to the numerical computation of orbits, since computer calculations can only produce pseudo-orbits (see \cite{TK, HYG}). For instance, if a computer is tasked with calculating the orbit of the point $1/6$, we observe a result akin to the figure depicted below (see Figure \ref{noisyOrbit}). However, it is well-known that $T(1/6)=1/2$.

\begin{figure}[h]
	\begin{center}
		\includegraphics[width=0.7\linewidth]{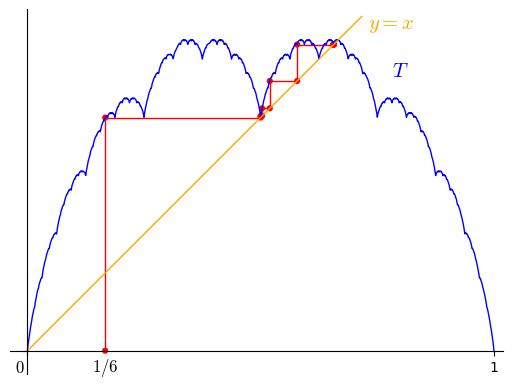}
		
		\caption{The orbit $(T^n(1/6))_n$ calculated by a computer, where $T\approx G_{13}$.}
		\label{noisyOrbit}
	\end{center}
\end{figure}

The previous figure shows that the computer-generated orbit differs significantly from the true orbit starting at $1/6$. However, given a computer-generated orbit, we may ask whether there exists a true orbit with a slightly different initial point that stays close to it for a long time. This leads to the study of the shadowing property: 

\begin{defn}
	A continuous function $f:\mathbb{R}\to \mathbb{R}$ is said to have the \emph{shadowing property} if for every $\varepsilon>0$ there exists $\delta>0$ such that for any $\delta$-pseudo orbit $(x_n)$ of $f$ there is $x^*\in \mathbb{R}$ such that 
	$$
	|f^n(x^*)-x_n|<\varepsilon\quad \text{ for all }\,n\geq 0.
	$$
\end{defn}

This concept was originally studied by D. Anosov (see \cite{An})  and R. Bowen (see \cite{Bo}), and has since become a fundamental notion in modern dynamical systems (see  \cite{Palmer} and \cite{Pilyugin}).

Sections \ref{BehaviorOrbitSection} and \ref{ShadowingPropertySection} are devoted to proving that the Takagi function has the shadowing property. In the former, we develop several technical results concerning the behavior of an orbit of the Takagi function starting very close to a fixed point. To this end, we introduce the concept of a point making a $(u,v,\eta)-jump$ of a certain order where $u$ and $v$ are fixed points of the Takagi function, $u,v$ and $\eta>0$ (see Definition \ref{jump}). In the latter, we combine these results to prove that the Takagi function has the shadowing property (see Theorem \ref{main}).

In 1988, E. M. Coven, I. Kan, and J. A. Yorke conducted a detailed study of the shadowing property for the family of tent maps. Inspired by this work, we introduce a family of maps, called the Takagi family, and study the shadowing property for this class of maps in Section \ref{NoShadowingSection}. It is worth mentioning that our approach reveals new geometrical properties of the Takagi function, like the existence of quasi-tangent points and their variants, that are of independent interest. 

\section{Fixed points and dynamical behavior}\label{S1}
We begin this section by determining the set of fixed points of the Takagi function, providing explicit expressions for these points in terms of their binary expansions. To achieve this, our approach consists of computing the set of fixed points of the function $G_{2n}$ and subsequently leveraging the fact that the functions $G_n$ converge uniformly to the Takagi function while approximating it monotonically from below. To this end, we use the following property of the function $g_n$, which is a particular case of Lemma 4.2 in \cite{FGGLl2}. For the sake of completeness, we provide a brief proof.
\begin{lemma}\label{2en2}
Let $x\in [0,1]$ with binary expansion given by $x=\sum_{n=1}^{\infty}\varepsilon_n 2^{-n}$. For every $n$, we have 
$$
g_n(x)+g_{n+1}(x)\leq 2^{-n}
$$
and the equality holds if and only if $\varepsilon_n+\varepsilon_{n+1}=1$.
\end{lemma}
\begin{proof}
It is clear that $g_1(x)+g_2(x)\leq 1/2$, and the equality holds if and only if $\varepsilon_1+\varepsilon_{2}=1$ (see Figure \ref{Fig_grafoG2} below). The functions $g_1$ and $g_2$ may be extended periodically to the whole real line with period one, so we have 
$$
g_n(x)+g_{n+1}(x)=\frac{1}{2^{n-1}}g_1\left (2^{n-1}x\right )+\frac{1}{2^{n-1}}g_2\left (2^{n-1}x\right )
$$
and observe that $2^{n-1}x=\varepsilon_1\cdots\varepsilon_{n-1}.\varepsilon_n\varepsilon_{n+1}\cdots$, which gives us the result.
\end{proof}

\begin{figure}[h]
	\begin{center}
		 \includegraphics[width=0.7\linewidth]{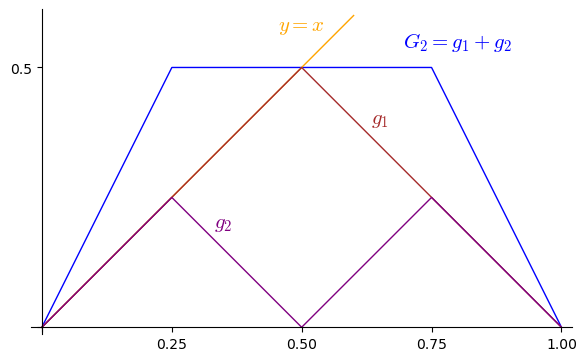}
		\caption{The graph of the function $G_2=g_1+g_2$}
		\label{Fig_grafoG2}
	\end{center}
\end{figure}

Let $z_0=0$. For every $n$, we consider the dyadic point
\begin{equation}\label{z_n}
z_n=\sum_{j=1}^{2n}\frac{\varepsilon_j}{2^j}=\sum_{k=1}^{n}\frac{1}{2^{2k-1}}=\frac{2}{3}\left ( 1-2^{-2n}\right )
\end{equation}
satisfying that $\varepsilon_{2k-1}=1$ and $\varepsilon_{2k}=0$ for  every $1\leq  k \leq  n$. Thus, the binary expansion of $z_n$ consists of $n$ blocks, each composed of the two digits `10', that is 
$$
z_n=0.\underbrace{1010\cdots 10}_{n\, \text{blocks}}
$$
\begin{lemma}\label{G_2n}
For every $n$, the function $G_{2n}$ has the following properties:
\begin{enumerate}
\item[(1)]  The set of fixed points of $G_{2n}$ is given by $\{z_0, z_1, \ldots, z_n \}$.
\item [(2)] The maximum value of $G_{2n}$ is $z_n$.
\item[(3)] For every $1\leq j \leq n$ we have $G_{2n}(y)>y$ for all $y\in (z_{j-1}, z_j)$.
\end{enumerate}
\end{lemma}
\begin{proof}
We proceed by induction. When $n=1$ we have $z_0=0$ and $z_1=1/2$. The maximum value of $G_{2}$ is $1/2$ by Lemma \ref{2en2}. Moreover, we have $G_2(0)=0$ and $G_2(1/2)=g_1(1/2)+g_2(1/2)=1/2$ by Lemma \ref{2en2} again. Note that $G_2(y)>y$ for all $y\in (0,1/2)$.

Now, assume that the result holds true for some  $n$ and we will prove it for $n+1$. As previously mentioned, the functions $g_1$ and $g_2$ can be extended periodically to the whole real line. Hence, for every $y\in  [0,1]$ we have $g_{2n+1}(y) = 2^{-2n}g_1(2^{2n}y)$ and $g_{2n+2}(y) = 2^{-2n}g_2(2^{2n}y)$, so we may write
\begin{equation}\label{eqAux1}
G_{2(n+1)}(y)= G_{2n}(y)+ g_{2n+1}(y)+g_{2n+2}(y) =G_{2n}(y)+ \frac{1}{2^{2n}}G_2(2^{2n}y). 
\end{equation}
The induction hypothesis yields that $G_{2(n+1)}(y)\leq z_n+ 2^{-(2n+1)} = z_{n+1}$ for all $y\in [0,1]$. Furthermore, by Lemma \ref{2en2} we obtain 
\begin{align*}
G_{2(n+1)}(z_{n+1})&=\sum_{k=1}^{2(n+1)}g_k(z_{n+1})=\sum_{j=1}^{n+1}g_{2j-1}(z_{n+1})+g_{2j}(z_{n+1})\\&=\sum_{j=1}^{n+1} \frac{1}{2^{2j-1}}=z_{n+1}.
\end{align*} 
This gives that the maximum value of $G_{2(n+1)}$ is $z_{n+1}$ and  we also conclude that the set of fixed points of $G_{2(n+1)}$ is given by $\{z_0, z_1, \ldots, z_{n+1} \}$ since $G_{2(n+1)}(z_j)= G_{2n}(z_j)$ for every $0\leq j \leq n$. Finally, for every $1\leq j \leq n$ we have $G_{2(n+1)}(y)\geq G_{2n}(y)\geq G_{2j}(y)>y$ for all $y\in (z_{j-1}, z_j)$. If $y\in(z_n, z_{n+1})$ then $0< 2^{2n}y<1/2$ and by \eqref{eqAux1} we obtain $G_{2(n+1)}(y)>y$. 
\end{proof}


\begin{thm}\label{fixedPoints}
The set of fixed points of the Takagi function is given by 
$$
\text{Fix}\, (T)= \left \{z_n:n\in\mathbb{Z}_{\geq 0}\right \}\cup \{2/3\}.
$$ 
Moreover, for every $n\in \mathbb{N}$ we have $T(y)>y$ for all $y\in (z_{n-1}, z_n)$.
\end{thm}
\begin{proof}
Recall that $T(0)=0$. For every $n$ we have $T(z_n)=G_{2n}(z_n)=z_n$ and  Lemma \ref{G_2n} gives that $T(y)\geq G_{2n}(y)>y$ for all $y\in (z_{n-1}, z_n)$. The sequence $(z_n)_n$ converges to $2/3$ and by the continuity of the Takagi function, it is also a fixed point. In order to see that there are no more fixed points, suppose that there is $y\in [0,1]\setminus \left (\left \{z_n:n\in\mathbb{Z}_{\geq 0}\right \}\cup \{2/3\}\right )$ such that $T(y)=y$. The maximum value of the Takagi function is $2/3$, so this implies that $y<2/3$.  Since $(z_n)$ is an increasing sequence converging to $2/3$, there is an indice $n\in \mathbb{N}$ such that $z_{n-1}<y<z_{n}$. However, by Lemma \ref{G_2n} again, we have $T(y)\geq G_{2n}(y)>y$ and this is a contradiction.
\end{proof}

After computing the set of fixed points, our next objective is to analyze the behavior of the orbits in the discrete dynamical system \eqref{dds}. This prompts us to introduce the following notation. Let $S$ be the set of points $x\in [0,1]$ satisfying that there are non-negative integers $n_x$ and $k$ such that $T^{n_x}(x)=z_k$. In other words, the set $S$ consists of those points whose orbit eventually lands at a fixed point different from $2/3$.
\begin{proposition}\label{dynamicBehavior}
Let $x\in [0,1]$.The following statements hold:
\begin{enumerate}
\item If $x\notin S$ then $T^n(x)\to 2/3$ as $n\to \infty$.
\item If $x\in D$, then $x\in S$.
\end{enumerate}
\end{proposition}
\begin{proof}
Since  $T^n(x)\leq 2/3$ for every $n$ and $T(y)\geq y$ for all $0\leq y\leq 2/3$, we have that $T^n(x)$ is a non-decreasing sequence bounded above and hence it converges to a certain point $p\in [0,2/3]$, which is necessarily a fixed point by continuity.

(1) If $x\notin S$ then $T^n(x)\neq z_k$ for every non-negative integers $n$ and $k$. Since  every point of the set $\left \{z_n:n\in\mathbb{Z}_{\geq 0}\right \}$ is a dyadic point and the Takagi function has a strict local minimum value at every dyadic point (see \cite{Ka} for a proof), it follows that $p=2/3$.

(2) It is enough to prove that if $x\in D_{k+1}$ for some $k$ then $T(x)\in D_{k+1}$ too. Indeed, this implies that the sequence $T^n(x)$ takes finitely many values, and hence it has to become constant eventually since it is a convergent sequence. If $x\in D_{k+1}$ then 
$
x=\sum_{j=1}^{k}\varepsilon_j2^{-j}
$
with $\varepsilon_j\in\{0,1\}$
and $g_j(x)=0$ for all $j\geq k+1$. Thus, we have
$$
T(x)=\sum_{j=1}^kg_j(x)=\sum_{j=1}^{k}\left ( (1-\varepsilon_j)\sum_{m=j}^k \frac{\varepsilon_m}{2^m}+ \varepsilon_j \left [2^{-(j-1)}- \sum_{m=j}^k \frac{\varepsilon_m}{2^m}\right ]\right )
$$
so it follows that $T(x)$ is a finite sum of dyadic numbers belonging to $D_{k+1}$, and hence it belongs to $D_{k+1}$ too.
\end{proof}

\begin{remark}
It is immediate to see that $T(1/6)=1/2$, so the set $S$ does not consist of dyadic numbers exclusively. 
\end{remark}

\begin{thm}\label{orbit}
For almost every point $x\in [0,1]$ we have $T^n(x)\to 2/3$ as $n\to \infty$.
\end{thm}
\begin{proof}
By Proposition \ref{dynamicBehavior} it suffices to show that the set $S$ has measure zero. Indeed, we write 
$$
S=\bigcup_{n \geq 1}\, \bigcup_{k\geq 0}\, S_{n,k}
$$
where $S_{n,k}=\{x\in [0,1]: T^n(x)=z_k\}$. Proceeding towards a contradiction, if $\lambda(S_{n,k})>0$ for some $n$ and $k$, then we pick $x_0\in S_{n,k}$ such that $x_0$ is a density point of $S_{n,k}$ and we have
$$
\lim_{z\to x_0,\, z\in S_{n,k}} \frac{T^n(z)-T^n(x_0)}{z-x_0}=0
$$
which gives that $T^n$ is approximately differentiable at $x_0$. This is a contradiction to the Corollary \ref{NAD1}. Therefore, $\lambda(S_{n,k})=0$ and this proves the result. 
\end{proof}

In view of the previous results, for a given point $x_0\in [0,1]$ either the orbit $\left(T^n(x_0)\right )_n$ converges  to $2/3$, or it eventually lands on a fixed point different from $2/3$. This fully characterizes the behavior of an arbitrary orbit of the discrete dynamical system \eqref{dds}.

\section{The Shadowing Property and the Takagi function: Behavior of a Real Orbit}\label{BehaviorOrbitSection}

This section focuses on the behavior of an orbit that starts very close to a fixed point $z_n$. First, we note a certain symmetry property exhibited by the Takagi function on the interval $[z_n-2^{-2n}, z_n+2^{-2n}]$. Indeed, for every $y\in [0,1]$ we have
\begin{align*}
T\left (z_n-\frac{y}{2^{2n}}\right )&=G_{2n}\left (z_n-\frac{y}{2^{2n}} \right )+\sum_{j=2n+1}^{\infty}g_j(z_n-\frac{y}{2^{2n}})\\&=G_{2n}\left (z_n+\frac{y}{2^{2n}} \right )+\sum_{j=2n+1}^{\infty}g_j(z_n+\frac{y}{2^{2n}})=T\left (z_n+\frac{y}{2^{2n}}\right )
\end{align*}
where we have used the fact that $G_{2n}$ is constant on the interval $[z_n-2^{-2n}, z_n+2^{-2n}]$ and that the function $g_j$ is symmetric with respect to $z_n$ for all $j\geq 2n+1$. Therefore, if $x\in [z_n-2^{-2n}, z_n]$ then $2z_n-x$ belongs to $[z_n, z_n+2^{-2n}]$ and we have $T(x)=T(2z_n-x)$. Throughout the paper, we shall sometimes refer to this property as the local symmetry of the Takagi function at the fixed point $z_n$. 

Let $0<\varepsilon<2/3$ be given. Since the sequence of fixed points $(z_n)_n$, defined in \eqref{z_n}, is stricly increasing and converges to $2/3$, there is $m\in \mathbb{N}$ such that 
\begin{equation}\label{def_m}
z_m=\frac{2}{3}\left (1-2^{-2m} \right ) \in (2/3- \varepsilon/2, 2/3) \quad \text{ and } z_{m-1}\leq 2/3- \varepsilon/2.
\end{equation}
\begin{lemma}\label{epsilonPrime}
Let $0<\varepsilon<2/3$ and $m\in \mathbb{N}$ defined as in \eqref{def_m}.  Then, there is $0<\varepsilon'<\varepsilon/4$ satisfying the following properties:
\begin{enumerate}[label=\small{(\roman*)}]
\item \label{away} If $x\in \left [z_{j}-\varepsilon', z_j +\varepsilon'\right]$ for some $j=0,\ldots, m$, then 
$$
T(x)\geq z_j+100^m\left |x-z_j\right |.
$$
\item \label{outEpsilonNeighbourhood} For every $0<\eta\leq \varepsilon'$, if $x\in [z_{j-1}+\eta, z_j -\eta]$ for some $j=1,\ldots, m$ then 
$$
T(x)\geq x+\eta.
$$
\item \label{notFar} If $x\in [z_{j}-\varepsilon', z_j +\varepsilon']$ for some $j=0,\ldots, m-1$, then $T(x)<z_{j+1}$.
\end{enumerate}
\end{lemma}
\begin{proof}
We choose 
$$
\varepsilon'= 2^{-\left (2m+100^m+2\right )}.
$$
It follows from \eqref{def_m} that
$$
\frac{8}{3}\cdot \varepsilon' < \frac{8}{3}\cdot 2^{-(2m+2)}= \frac{2}{3}\cdot 2^{-2m}=\frac{2}{3}-z_m<\frac{\varepsilon}{2}
$$
which yields $\varepsilon'<\varepsilon/4$. Furthermore, for every $j=0,\ldots, m$ the polygonal function $G_{2m+100^m}$ is affine on $[z_j-\varepsilon',z_j]$ and $[z_j,z_j+\varepsilon']$ with 
\begin{equation}\label{bigSlope}
G'^+_{2m+100^m}(z_j)=-G'^-_{2m+100^m}(z_j)\geq 100^m
\end{equation}
and hence if $x\in [z_{j}-\varepsilon', z_j +\varepsilon']$ for some $j=0,\ldots, m$, then 
$$
T(x)\geq G_{2m+100^m}(x)\geq z_j+100^m|x-z_j|,
$$
which yields \ref{away}. 

Now, we prove \ref{outEpsilonNeighbourhood}. Observe that $G'^+_{2m+2}(z_j)=-G'^-_{2m+2}(z_j)\geq 2$ for every $j=0,\ldots, m$. As we have $\varepsilon'<2^{-(2m+2)}$, it is immediate to see that for every $0<\eta\leq \varepsilon'$  each point  of the interval $[0, z_m]$ where the line $y=x+\eta$ intersects the graph of the polygonal function $G_{2m+2}$ belongs to $[z_{j}-\eta, z_j +\eta]$ for some $0\leq j \leq m$, and if $x\in [z_{j-1}+\eta, z_j -\eta]$ for some $j=1,\ldots, m$ then  $G_{2m+2}(x)>x+\eta$, which gives the result.

Finally, we prove \ref{notFar}. Let $j=0,\ldots,m-1$. We have $G'^+_{2j}(z_j)=0$,  and it follows from \eqref{selfAffineProperty} that 
\begin{equation}\label{sapEpsilon'}
T\left (z_j+\frac{y}{2^{2j}}\right )=z_j+\frac{1}{2^{2j}}T(y)
\end{equation}
for every $y\in [0,1]$. Moreover, observe that $z_{j+1}=z_j+2^{-(2j+1)}$. Since 
$$
z_j+\varepsilon'<z_j+ \frac{1}{2^{2m+2}}<z_j+\frac{1}{6}\cdot \frac{1}{2^{2j}},
$$
for $x\in [z_j, z_j+\varepsilon']$ we may write 
$$
x=z_j+\frac{x_0}{2^{2j}}
$$
for some $x_0\in [0,1/6)$. By Lemma  \ref{line} together with \eqref{sapEpsilon'} we obtain 
$$
T(x)=T\left (z_j+\frac{x_0}{2^{2j}}\right )=z_j+\frac{1}{2^{2j}}T(x_0)<z_j+\frac{1}{2}\cdot \frac{1}{2^{2j}}=z_{j+1}
$$
The result follows immediately from the local symmetry of the Takagi function at each $z_n$. 
\end{proof}

\begin{proposition}\label{counting}
Let $\varepsilon>0$. Let $m\in \mathbb{N}$ be defined as in \eqref{def_m} and $\varepsilon'>0$ be given by Lemma \ref{epsilonPrime}. Then, we have the following:
\begin{enumerate}[label=\small{(\roman*)}]
\item \label{OutOfEpsilon'} If $x\in \left [z_{j-1}+\varepsilon',  z_j-\varepsilon'\right ]$ for some $j=1,\ldots, m$, then there is $K\in \mathbb{N}$ such that $T^{K-1}(x)< z_j$ and  $T^{K}(x)\geq z_j$ with 
$$
K\leq \frac{2/3}{\varepsilon'}+1.
$$
\item\label{OutOfEta} For every $0<\eta \leq \varepsilon'$, if $x\in \left [z_j+\eta, z_j+\varepsilon'\right ]$ for some $j=1,\ldots, m-1$, then there is $L\in \mathbb{N}$ such that $T^{L-1}(x)\leq  z_j+\varepsilon'$ and  $T^{L}(x)> z_j+\varepsilon'$ with 
$$
L\leq \frac{-\log \eta}{m\cdot \log 100}+1.
$$ 
\end{enumerate}
\end{proposition}
\begin{proof}
First, we prove \ref{OutOfEpsilon'}. Proceeding by induction, we show that 
\begin{equation}\label{ind1}
T^{l}(x)\geq z_{j-1}+ (l+1)\cdot \varepsilon'
\end{equation}
for every $l\in \mathbb{N}$ such that $T^{l-1}(x)\leq  z_j-\varepsilon'$. Indeed, if $l=1$ then by  \ref{outEpsilonNeighbourhood} in Lemma \ref{epsilonPrime} used with $\eta=\varepsilon'$ we have
$$
T(x)\geq x+\varepsilon'>z_{j-1}+2\varepsilon'.
$$
Now, assume that the property holds for some $l\geq 1 $ and  $T^{l}(x)\leq z_j-\varepsilon'$. We will prove it for $l+1$. We have 
\begin{align*}
T^{l+1}(x)=T\left (T^{l}(x)\right )&\geq T^{l}(x)+\varepsilon' \geq z_{j-1}+(l+2)\cdot\varepsilon'
\end{align*} 
where we used that $z_{j-1}+\varepsilon'\leq x<T^{l}(x)$ and \ref{outEpsilonNeighbourhood} in Lemma \ref{epsilonPrime}  again.  

It follows from \eqref{ind1} that there is $l_0\in \mathbb{N}$ such that $T^{l_0}(x)>z_j-\varepsilon'$ and $T^{l_0-1}(x)\leq  z_j-\varepsilon'$. By \eqref{ind1} again we have 
$$
z_{j-1}+ l_0\cdot \varepsilon'\leq T^{l_0-1}(x)\leq z_j-\varepsilon' < \frac{2}{3}
$$
which implies $l_0\leq \frac{2/3}{\varepsilon'}$. Finally, if $T^{l_0}(x)\geq z_j$ then we take $K=l_0$. Otherwise we have $T^{l_0+1}(x)\geq z_j$ by \ref{away} in Lemma \ref{epsilonPrime}, so we take $K=l_0+1$.

Our next task is to prove \ref{OutOfEta}. Proceeding by induction as we did before, and taking advantage of \ref{away} in Lemma \ref{epsilonPrime}, we have that 
\begin{equation}\label{ind2}
T^{l}(x)\geq z_j + \eta\cdot 100^{ml}
\end{equation}
for every $l\in \mathbb{N}$  such  that $ T^{l-1}(x)\leq z_j+\varepsilon'$. This implies that there is $l_1\in \mathbb{N}$ such that $T^{l_1}(x)>z_j+\varepsilon'$ and $T^{l_1-1}(x)\leq  z_j+\varepsilon'$. By \eqref{ind2} again we obtain
$$
z_j + \eta\cdot 100^{m(l_1-1)}\leq T^{l_1-1}(x)\leq  z_j+\varepsilon'<1
$$
and hence
$$
l_1\leq \frac{-\log \, \eta}{m\cdot \log\, 100}+1
$$
which gives us the result.


\end{proof}

%


We are interested in determining the behavior of the orbit of a point that is very close to a fixed point. To describe this behavior, we introduce a definition that will play a crucial role in the subsequent development.
\begin{defn}\label{jump}
Let $\eta>0$, $N\in \mathbb{N}$ and $u,v\in \text{Fix}\, (T)$ with $u<v$. A point $x\in [u-\eta, u+\eta]$ makes a $(u,v,\eta)$-\emph{jump} of order $N$ if 
\begin{enumerate}[itemsep=0.1em,label=\small{(\roman*)}]
\item $T^{N}(x), T^{N+1}(x)\in [v-\eta, v+\eta]$, and 
\item for every $z\in \text{Fix}\, (T)\setminus \{u,v\}$ we have $$\left \{T^n(x), T^{n+1}(x)\right \}\not \subset \left [z-\eta, z+\eta\right ]$$
provided that $n\leq N$.
\end{enumerate}
\end{defn}

\begin{proposition}\label{behavior}
Let $\varepsilon>0$. Let $m\in \mathbb{N}$ be defined as in \eqref{def_m} and $\varepsilon'>0$ be given by Lemma \ref{epsilonPrime}.  Let $0<\eta<\varepsilon'$. If $x\in [z_j-\eta, z_j+\eta]$ for some $j=0,\ldots, m-1$ and $T^2(x)> z_j+\eta$, then there exists $N=N(x)\in \mathbb{N}$ such that
$$
N\leq \frac{-\log\, \eta }{\log \, 100} +3m+ \frac{2/3}{\varepsilon'}m
$$
and  exactly one of the following holds: 
\begin{enumerate}[label=\small{(\roman*)}]
       \item\label{case1}  We have $T^{N-1}(x)< z_m$ and $T^N(x)\geq z_m$.
	\item \label{case2} There is an index $j<\omega\leq m-1$ such that $x$ makes a $(z_j,z_\omega,\eta)$-\emph{jump} of order $N$.
\end{enumerate}
\end{proposition}
\begin{proof}
We define by induction two finite sequences of indices $n_1,\ldots, n_{p}$ and $j_1,\ldots,j_p$ for some $1\leq p\leq m$. If $T(x)> z_j+\eta$ then we take $x^*=x$, otherwise we take $x^*=T(x)$. Thus, we have $x^*\in [z_j-\eta, z_j+\eta]$ and $T(x^*)>z_j+\eta$. Observe that by \ref{notFar} in Lemma \ref{epsilonPrime}, $T(x^*)<z_{j+1}$. It follows from Proposition \ref{counting} and \ref{away} in Lemma \ref{epsilonPrime} that there is $l_1\in \mathbb{N}$ with 
\begin{equation}\label{proc0}
l_1\leq \frac{-\log \eta}{m\cdot \log 100}+2+\frac{2/3}{\varepsilon'}+1
\end{equation}
such that $T^{l_1}(x)\geq z_{j+1}$ and $T^{l_1-1}(x)<z_{j+1}$. 

If $j+1= m$ then the result follows immediately with $\omega=j+1$ and $N=l_1$, and we let $p=0$. 

Otherwise we have $j+1<m$. 

If $\left \{ T^{l_1-1}(x), T^{l_1}(x)\right \}\subset [z_{j+1}-\eta,z_{j+1}+\eta ]$ then $x$ makes a $(z_j, z_{j+1}, \eta)$-jump of order $l_1-1$ and the result follows with $N=l_1-1$ and $\omega=j+1$, and we let $p=0$. 

Otherwise, if $\left \{ T^{l_1}(x), T^{l_1+1}(x)\right \}\subset [z_{j+1}-\eta,z_{j+1}+\eta ]$ then $x$ makes a $(z_j, z_{j+1}, \eta)$-jump of order $l_1$ and the result follows with $N=l_1$ and $\omega=j+1$, and we also let $p=0$. 

If none of the previous cases apply,  then  $x$ does not make a $(z_j, z_{j+1}, \eta)$-jump of any order and we have either $T^{l_1}(x)>z_{j+1}+\eta$ or $T^{l_1+1}(x)>z_{j+1}+\eta$, and in fact, in both cases we have $T^{l_1+1}(x)>z_{j+1}+\eta$. We let $j_1=j+1$ and we denote by $n_1$ the index such that $T^{n_1}(x)>z_{j_1}+\eta$ and $T^{n_1-1}(x)\leq z_{j_1}+\eta$. Observe that $n_1\leq l_1+1$.

Now, suppose that $n_k$ and $j_k$ have been defined for some $k\geq 1$ satisfying that $j_k<m$,  $T^{n_k}(x)>z_{j_k}+\eta$ and $T^{n_k-1}(x)\leq z_{j_k}+\eta$. It follows from \ref{notFar} in Lemma \ref{epsilonPrime} that $T^{n_k}(x)<z_{j_k+1}$. 

As before, it follows from Proposition \ref{counting} and \ref{away} in Lemma \ref{epsilonPrime} that there is $l_{k+1}\in \mathbb{N}$ with 
\begin{equation}\label{proc1}
l_{k+1}\leq \frac{-\log \eta}{m\cdot \log 100}+2+\frac{2/3}{\varepsilon'}
\end{equation}
such that $T^{n_k+l_{k+1}}(x)\geq z_{j_k+1}$ and $T^{n_k+l_{k+1}-1}(x)<z_{j_k+1}$. 

If $j_k+1= m$ then the result follows with $\omega=j_k+1$ and $N=n_k+l_{k+1}$, and we let $p=k$. 

Otherwise we have $j_k+1<m$. 

If $\left \{ T^{n_k+l_{k+1}-1}(x), T^{n_k+l_{k+1}}(x)\right \}\subset [z_{j_k+1}-\eta,z_{j_k+1}+\eta ]$ then the point $x$ makes a $(z_j, z_{j_k+1}, \eta)$-jump of order $n_k+l_{k+1}-1$ and the result follows with $N=n_k+l_{k+1}-1$ and $\omega=j_k+1$, and we let $p=k$.

 Otherwise, if $\left \{ T^{n_k+l_{k+1}}(x), T^{n_k+l_{k+1}+1}(x)\right \}\subset [z_{j_k+1}-\eta,z_{j_k+1}+\eta ]$ then $x$ makes a $(z_j, z_{j_k+1}, \eta)$-jump of order $n_k+l_{k+1}$ and the result follows with $N=n_k+l_{k+1}$ and $\omega=j_k+1$, and we also let $p=k$. 

If none of the previous cases apply,  then  $x$ does not make a $(z_j, z_{j_k+1}, \eta)$-jump of any order and we have either $T^{n_k+l_{k+1}}(x)>z_{j_k+1}+\eta$ or $T^{n_k+l_{k+1}+1}(x)>z_{j_k+1}+\eta$. We let $j_{k+1}=j_k+1$ and we denote by $n_{k+1}$ the index such that $T^{n_{k+1}}(x)>z_{j_{k+1}}+\eta$ and $T^{n_{k+1}-1}(x)\leq z_{j_{k+1}}+\eta$. Finally, note that 
\begin{equation}\label{proc2}
n_{k+1}\leq n_k+l_{k+1}+1 \leq k+1+ l_1+l_2+\cdots + l_k+l_{k+1}.
\end{equation}
This process can be repeated $p$ times, with $p\leq m-1$. As a consequence, we obtain $N\in \mathbb{N}$ such that \ref{case1} or \ref{case2} holds and by using \eqref{proc0}, \eqref{proc1} and \eqref{proc2} we have
\begin{align*}
N&\leq n_p+l_{p+1} \leq p+l_1+\cdots+l_{p+1}\\&\leq p+1-\frac{\log \eta}{m\cdot \log 100}+2+\frac{2/3}{\varepsilon'}+l_2+\cdots+l_{p+1}\\&\leq   m+ \frac{-\log\, \eta }{\log \, 100} +2m+ \frac{2/3}{\varepsilon'}m.
\end{align*}
\end{proof}

\begin{lemma}\label{fix}
Let $\varepsilon>0$ and $z\in \text{Fix}\, (T)$. If $x\in [0,1]$ is such that $T^{N-1}(x)=T^{N}(x)=z$ for some $N\in \mathbb{N}$, then there exist $0<r<\varepsilon$ and $\eta>0$ satisfying
\begin{enumerate}[label=\small{(\roman*)}]
	\item $\lambda \left (T^n \big ([x-r, x+r]\big )\right )<\varepsilon$ for every $ n=0,\ldots, N$, and
	\item $[z, z+\eta]\subset T^N ([x-r, x+r])$.
\end{enumerate}
Therefore, there is a closed interval $I\subset [x-r, x+r]$ such that $T^{N}(I)=[z, z+\eta]$, $\text{dist}\, (T^n(x), T^n(I))<\varepsilon$ and $\lambda\left (T^n(I)\right )<\varepsilon$ for all $n=0, \ldots, N$.
\end{lemma}
\begin{proof}
For every $n=0, \ldots, N$ the function $T^n$ is continuous at $x$, so we can find $\delta_n>0$ satisfying that  $\delta_n<\varepsilon/2$ and $$T^n\big ([x-\delta_n, x+\delta_n]\big )\subset [T^n(x)-\varepsilon/2, T^n(x)+\varepsilon/2].$$
We set $r = \min \{\delta_n: n=0, \ldots, N\}$. Thus,  (1) holds. Moreover, for every $n=0, \ldots, N$ the function $T^n$ is constant on no interval of positive length since by Corollary \ref{NAD1} it is nowhere approximately differentiable, and hence $T^{N-1}\big ([x-r, x+r]\big )$ is a closed interval of positive length which contains $T^{N-1}(x)=z$. Since $z$ is a strict local minimum of the Takagi function,  we have that  $T^{N}\big ([x-r, x+r]\big )$ contains a closed interval of positive length whose left-endpoint is $z$. Therefore, we may take $\eta>0$ such that (2) holds.

Finally, since the function $T^N$ is continuous on $[x-r, x+r]$ and $[z, z+\eta]\subset T^N \big ([x-r, x+r]\big )$ there is a closed interval $I\subset [x-r, x+r]$ such that $T^{N}(I)=[z, z+\eta]$. The last statement follows since $I\subset [x-r, x+r]$ and $T^n$ is continuous on $[x-r, x+r]$ for every $n=0,\ldots,N$.
\end{proof}

\begin{lemma}\label{InfinitelyPreimages}
Let $j\in \mathbb{N}$ and $x\in (z_{j-1},z_j)$. Then, there is a strictly decreasing sequence $(x_n)_n$ with $x_1=x$ satisfying that 
$$
T(x_{n+1})=x_n\quad \text{ and } \quad \lim_{n\to \infty}x_n=z_{j-1}.
$$
\end{lemma}
\begin{proof}
Let $x_1=x$. We inductively define the sequence $(x_n)_n$ such that for each $n\in \mathbb{N}$, the element $x_{n+1}$ satisfies $T(x_{n+1})=x_n$ and $T(x)<x_n$ for all $x\in (z_{j-1}, x_{n+1})$. This is equivalent to saying that $x_{n+1}$ is the minimal element of $T^{-1}(x_{n})$ in $(z_{j-1}, x_{n})$. By Theorem \ref{fixedPoints} we know that $T(y)>y$ for all $y \in (z_{j-1},z_j)$, so the sequence $(x_n)_n$ is strictly decreasing. Furthermore, we have $\lim_{n\to \infty}x_n=z_{j-1}$ since $T(x_{n+1})=x_n$ for every $n$ and there are no fixed points in the interval $(z_{j-1},z_{j})$.
\end{proof}

\begin{thm}\label{Lemmam_0}
 Let $\varepsilon>0$ and we consider $\varepsilon'>0$ and $m\in \mathbb{N}$ as in \eqref{def_m}. If $0\leq \alpha < \omega \leq m$, then there is $\widetilde{\eta}=\widetilde{\eta}(\varepsilon', z_\alpha,z_\omega)>0$ such that if $\eta<\widetilde{\eta}$ and $x\in [z_\alpha-\eta,z_\alpha+\eta ]$ makes a $(z_\alpha,z_\omega,\eta)$-jump of order $N\in \mathbb{N}$, then there exists $\widetilde{x}\in [z_\alpha, T(x)]$ with $T^N(\widetilde{x})=z_\omega$ satisfying that
$$
\left |T^n(x)-T^n(\widetilde{x})\right |<\varepsilon' 
$$
for all $n=0,\ldots, N, N+1$.
\end{thm}
\begin{proof}
For the sake of contradiction, we suppose that Theorem  \ref{Lemmam_0} does not hold. Therefore, we may take a counterexample such that $\omega-\alpha$ is minimal. Let $r<\varepsilon'/4$. We choose $0<r_0<r$ such that $T(z_\alpha+r_0)=z_\alpha+r$ and $T(x)<z_\alpha+r$ for all $x\in (z_\alpha, z_\alpha+r_0)$. In other words, $z_\alpha+r_0$ is the least preimage of $z_\alpha+r$ which is larger  than $z_\alpha$.

It is easy to see that if $\eta<r_0$ and $x$ makes a $(z_\alpha,z_\omega,\eta)$-\emph{jump} then there is $n\in \mathbb{N}$ such that $T^n(x)\in [z_\alpha+r_0, z_\alpha+r ]$. We select $n(r,x)$ such that 
$$
T^{\,n(r,x)}(x)\in [z_\alpha+r_0, z_\alpha+r ]
$$
and $T^l(x)>z_\alpha+r$ for all $l> n(r,x)$.

Since Theorem  \ref{Lemmam_0} does not hold, there is a positive and strictly decreasing sequence $(\eta_k)_k$ converging to zero satisfying that for every $\eta_k$ there is a point $x_{k}\in [z_\alpha-\eta_k, z_\alpha+\eta_k]$ making a $(z_\alpha,z_\omega,\eta_k)$-\emph{jump} of order $N_k\in \mathbb{N}$ and for every $y\in [z_\alpha, T(x_k)]$ such that $T^{N_k}(y)=z_\omega$ there exists $n_y\in \{0,\ldots, N_k, N_k+1\}$ for which
\begin{equation}\label{epskr}
|T^{n_y}(y)- T^{n_y}(x_{k})|\geq \varepsilon'. 
\end{equation}
We may assume that $x_{k}\in [z_\alpha, z_\alpha+\eta_k]$ for every $k\in \mathbb{N}$. Furthermore, we can choose $n(r,x_{k})$ as above and we denote 
$$
x^*_{k}= T^{\,n(r,x_{k})}(x_{k})\in [z_\alpha+r_0, z_\alpha+r ].
$$
Observe that Theorem \ref{fixedPoints} ensures that the sequence $\left (T^l(x_k)\right )_l$ for ${l=0,\ldots, N_k}$ is monotone increasing, and hence we have 
\begin{equation}\label{*eqz0}
T^l(x_k)\leq x_k^*\leq z_\alpha+r<z_\alpha+\varepsilon'/4
\end{equation}
for all $l=0,\ldots, n(r,x_k)$. By compactness, and by turning to a subsequence if necessary, we may assume that $\left (x^*_{k}\right)_k$ converges to $x^*\in [z_\alpha+r_0, z_\alpha+r]$.

Now, suppose that there is $\delta>0$ such that for all $\alpha<j<\omega$ we have 
\begin{equation}\label{C1}
|T^n(x_{k})-z_j|>\delta
\end{equation}
for every $k=1,2, \ldots$ and for every $n(r,x_{k})\leq n\leq N_k$. Thus, the sequence of integers $\left ( N_k-n(r,x_{k})\right )_k$ is bounded. Again, by turning to a subsequence we may assume that $\left ( N_k-n(r,x_{k})\right )_k$ is constant. Observe that
$$
T^{N_k-n(r,x_{k})}(x^*_{k}) = T^{N_k}(x_{k})\in \left [z_\omega-\eta_k, z_\omega+\eta_k \right].
$$
Given that the sequence $(\eta_k)_k$ converges to zero and $(x^*_{k})_k$ converges to $x^*$, it follows by continuity that
$$
T^{N_k-n(r,x_{k})}(x^*)=z_\omega.
$$ 
Moreover, as $\left (x^*_{k}\right )_k$ converges to $x^*$, there is $k_0\in \mathbb{N}$ such that if $k\geq k_0$ then 
$$
|T^{l}(x^*_{k})-T^l(x^*)|<\varepsilon'
$$
for every $0\leq l\leq N_k-n(r,x_{k})$. There is $m_0\leq n (r,x_{k_0})+1 $ such that $T^{m_0-1}(x_{k_0})<x^*$ and $T^{m_0}(x_{k_0})\geq x^*$. We claim that there is a point $\widetilde{x}_{k_0}\in [z_\alpha, T(x_{k_0})]$ such that $T^{\,n (r,x_{k_0})} (\widetilde{x}_{k_0})=x^*$ and $z_\alpha<T^l\left (\widetilde{x}_{k_0}\right )\leq x^*\leq z_\alpha+r$ for every $l=0, \ldots, n\left (r,x_{k_0}\right )$. 

Indeed, if $m_0=n(r,x_{k_0})+1$ then
$$
x^*\in \left [ T^{\,n (r,x_{k_0})}(x_{k_0}), T^{\,n (r,x_{k_0})+1}(x_{k_0})\right ]\subset T^{\,n (r,x_{k_0})}\left ([x_{k_0}, T(x_{k_0}) ] \right )
$$
and hence there is $\widetilde{x}_{k_0}\in [x_{k_0}, T(x_{k_0}) ]$ such that $T^{\,n (r,x_{k_0})} (\widetilde{x}_{k_0})=x^*$. 

Otherwise, by Lemma \ref{InfinitelyPreimages}, we cand find $z_\alpha<x'_{k_0}<x_{k_0}<T\left (x_{k_0}\right)$ such that $T(x'_{k_0})=x_{k_0}$ and we have 
$$
x^*\in \left [T^{m_0-1}(x_{k_0}), T^{m_0}(x_{k_0}) \right ]\subset T^{m_0}\left ( [x'_{k_0},x_{k_0}]\right ),
$$
which yields that there exists $y_0\in [x'_{k_0}, x_{k_0}]$ such that $T^{m_0}(y_0)=x^*$. By Lemma \ref{InfinitelyPreimages} again, we can find a point $\widetilde{x}_{k_0}\in [z_\alpha, T(x_{k_0})]$ such that $T^{\,n (r,x_{k_0})} (\widetilde{x}_{k_0})=x^*$. Therefore, we have 
$$
T^{N_{k_0}}\left ( \widetilde{x}_{k_0}\right )=z_{\omega}\quad \text{ and }\quad |T^{l}(x_{k_0})-T^l(\widetilde{x}_{k_0})|<\varepsilon'
$$
for every $l=0, \ldots, N_{k_0}$, which contradicts \eqref{epskr}.

It follows from the above argument that property \eqref{C1} does not hold. This implies that we can select $\alpha<j_0<\omega$  such that for a suitable sequence $(m_{k,j_0})_k$ we have that
\begin{equation}\label{x0k}
\lim_{k\to \infty} \left |T^{\,m_{k,j_0}}(x^*_{k})-z_{j_0}\right |=0.
\end{equation} 
We may also assume that $j_0\in (\alpha, \omega)$ is minimal with this property, so there is $\delta_1>0$ such that for every $\alpha<l<j_0$ we have 
$$
|T^n(x^*_{k})-z_l|>\delta_1
$$
for every $k=1,2,\ldots, $ and for every $0\leq n \leq m_{k,j_0}$. Thus, there is $\overline{m}_{k,j_0}\leq m_{k,j_0}+1$ such that $T^{\,\overline{m}_{k,j_0}}(x^*_{k})>z_{j_0}$ and $T^{\,\overline{m}_{k,j_0}-1}(x^*_{k})\leq z_{j_0}$. Now we proceed similarly to what we did above. The sequence $(\overline{m}_{k,j_0})_k$ is bounded, and by turning to a subsequence we may assume that  $(\overline{m}_{k,j_0})_k$ is constant. Then, \eqref{x0k} implies that $T^{\,\overline{m}_{k,j_0}}(x^*)=z_{j_0}$. Hence, there is $k_1\in \mathbb{N}$ such that if $k\geq k_1$ then 
\begin{equation}\label{x0kl}
|T^{n}(x^*_{k})-T^n(x^*)|<\varepsilon'/2
\end{equation}
for every $0\leq n\leq \overline{m}_{k,j_0}$. By applying Lemma \ref{fix} to $x^*$, we obtain $0<r^*<\varepsilon'/2$ and $\eta^*>0$, as well as a closed interval $I\subset [x^*-r^*, x^*+r^*]$ such that 
\begin{equation}\label{L11}
T^{\,\overline{m}_{k,j_0}+1}(I)=[z_{j_0}, z_{j_0}+\eta^*],
\end{equation} 
 $\text{dist}\, (T^n(x^*), T^n(I))<\varepsilon'/2$ and $\lambda\left (T^n(I)\right )<\varepsilon'/2$ for all $n=0, \ldots, \overline{m}_{k,j_0}+1$.  By continuity, there exists $\rho>0$ such that if $|y-z_{j_0}|<\rho$ then $|T(y)-z_{j_0}|<\eta^*$.

Since $\omega-j_0<\omega-\alpha$ and the counterexample was chosen so that $\omega - \alpha$ is minimal, we may apply Theorem \ref{Lemmam_0} for $z_{j_0}$ and $z_\omega$. Therefore, there is $0<\widetilde{\eta}\,'<\min \{\varepsilon'/2, \eta^*,\rho\}$ such that if $\eta<\widetilde{\eta}\,'$ and $y$ makes a $(z_{j_0},z_\omega,\eta)$-\emph{jump} of order $N'\in \mathbb{N}$, then there exists $\widetilde{y}\in [z_{j_0}, T(y)]$ satisfying that
$$
T^{N'}(\widetilde{y})=z_\omega \quad \text{ and } \quad |T^n(y)-T^n(\widetilde{y})|<\varepsilon' \text{ for all }n=0,\ldots, N', N'+1.
$$

Now, we choose $k_2\geq k_1$ such that $\eta_{k_2} < \widetilde{\eta}\,'/2$ and
\begin{equation}\label{new}
T^{\,\overline{m}_{k_2,j_0}}(x^*_{k_2})\in (z_{j_0}, z_{j_0}+\widetilde{\eta}\,'/2).
\end{equation}
Thus, $T^{\,\overline{m}_{k_2,j_0}}(x^*_{k_2})$ makes a $(z_{j_0}, z_\omega,\widetilde{\eta}\,'/2)$-jump of order $N_{k_2}-\overline{m}_{k_2,j_0}-n(r,x_{k_2})$. Since \eqref{new} implies $T^{\,\overline{m}_{k_2,j_0}+1}(x^*_{k_2})<z_{j_0}+\eta^*$, as a result of applying Theorem \ref{Lemmam_0} to $T^{\,\overline{m}_{k_2,j_0}}(x^*_{k_2})$, we find a point $\overline{y}\in [z_{j_0}, z_{j_0}+\eta^*]$ such that 
\begin{equation}\label{eqi}
T^{N_{k_2}-\overline{m}_{k_2,j_0}-n(r,x_{k_2})}(\overline{y})=z_\omega \quad \text{ and }\quad \left |T^n(\overline{y})-T^{n+\overline{m}_{k_2,j_0}}(x^*_{k_2})\right |<\varepsilon'
\end{equation}
 for all $n=0,\ldots, N_k-\overline{m}_{k,j_0}-n(r,x_{k_2})+1$. From \eqref{L11} we can also find a point $\widetilde{y}\in I \subset  [x^*-r^*, x^*+r^*]$ such that $T^{\, \overline{m}_{k_2,j_0}}(\widetilde{y})=\overline{y}$ and by using \eqref{x0kl} we obtain that
\begin{equation}\label{eqii}
 \begin{aligned}
 |T^n(\widetilde{y})-T^n(x^*_{k_2})|&\leq \left |T^n(\widetilde{y})- T^n(x^*)\right |+\left |T^n(x^*_{k_2})- T^n(x^*)\right |\\&<\varepsilon'/2+\varepsilon'/2=\varepsilon'
 \end{aligned}
\end{equation}
 for all $n=0, \ldots, \overline{m}_{k_2,j_0}$. By Lemma \ref{InfinitelyPreimages} it follows that we can choose a point $y$ such that $T^{n(r,x_{k_2})}(y)=\widetilde{y}$ and $$z_\alpha<T^{\, l}(y)\leq\widetilde{y}\leq x^*+r^*\leq  z_\alpha+r+r^*<z_\alpha+\varepsilon'$$ for every $l=0, \ldots, n(r,x_{k_2})$. By \eqref{*eqz0} we have that 
$$
\left |T^{\, l}(y)-T^{\, l}(x_{k_2})\right|<\varepsilon'
$$
 for every $l=0, \ldots, n(r,x_{k_2})$. This, together with \eqref{eqi} and \eqref{eqii} yields
$$
|T^n(y)-T^n(x_{k_2})|<\varepsilon'
$$ 
for every $n=0, \ldots, N_{k_2}$. This contradicts \eqref{epskr}. 
\end{proof}

\section{The Takagi function has the shadowing property}\label{ShadowingPropertySection}

Building on the results obtained previously, we dedicate this section to proving that the Takagi function has the shadowing property. We begin with some technical results.

\begin{lemma}\label{expansive}
Let $I$ be an open interval in $[0,2/3]$. Then, every connected component of $T^{-1}(I)$ has length at most $6\lambda(I)$.
\end{lemma}
\begin{proof}
Let $(a,b)$ be a connected component of $T^{-1}(I)$ and let $n$ be an integer such that $2^{-n}<b-a\leq 2^{-(n-1)}$. Therefore, there exists $v\in (a,b)\cap D_{n+1}$. We denote $a_n=v-2^{-(n+1)}$ and $b_n=v+2^{-(n+1)}$. 

We have either $a_n\in (a,b)$ or $b_n\in (a,b)$ because otherwise $b-a\leq 2^{-n}$. We suppose that $b_n\in (a,b)$, and the other case is similar. The function $G_{n+1}$ is affine on $(v,b_n)$, Moreover, $G_{n+1}(v)=T(v)$ and  $G_{n+1}(b_n)=T(b_n)$ since $b_n\in D_{n+2}$. If $G'^{+}_{n+1}(v)=G'^{-}_{n+1}(b_n)=0$, then by \eqref{selfAffineProperty} we get
$$
T\left (v+\frac{2}{3}\cdot \frac{1}{2^{n+1}} \right ) = T(v)+\frac{1}{2^{n+1}} T(2/3)=T(v)+\frac{2}{3}\cdot \frac{1}{2^{n+1}}  
$$
which implies $\lambda(I)\geq \frac{2}{3}\cdot \frac{1}{2^{n+1}}$. Otherwise we have $|G'^{+}_{n+1}(v)|\geq 1$ and hence
$$
|T(b_n)-T(v)|=|G_{n+1}(b_n)-G_{n+1}(v)|\geq b_n-v=\frac{1}{2^{n+1}}
$$
which implies $\lambda(I)\geq 2^{-(n+1)}$. We conclude 
$$
b-a\leq \frac{1}{2^{n-1}}=6\cdot \frac{2}{3}\cdot \frac{1}{2^{n+1}}\leq 6\lambda(I).
$$
\end{proof}
Thanks to the uniform continuity of the Takagi function we obtain the following result:

\begin{lemma}\label{delta_eta}
Let  $\rho>0$ and $M\in \mathbb{N}$. Then, there is $0<\delta<\rho$ such that if $\{x_0,x_1,x_2,\ldots\}$ is a $\delta$-pseudo orbit, then for any given index $j\geq 0$ we have
	$$
	|T^l(x_j)-x_{j+l}|<\rho
	$$
for  $l=0,1,\ldots, M$.
\end{lemma}
\begin{proof}
Since the Takagi function is uniformly continuous, there is a $\delta_0>0$ with $\delta_0<\rho$ such that $|T(x)-T(y)|<\alpha$ whenever $|x-y|<\delta_0$. Again, by the uniform continuity, there is $\delta_1>0$ with $\delta_1<\delta_0/2$ such that $|T(x)-T(y)|<\delta_0/2$ whenever $|x-y|<\delta_1$. By repeating this process, we obtain a sequence $\delta_0,\delta_1, \ldots, \delta_M$ satisfying that
	$$
	\rho>\delta_0>\frac{\delta_0}{2}>\delta_1>\frac{\delta_1}{2}>\delta_2>\cdots>\frac{\delta_{M-1}}{2}>\delta_{M}
	$$ 
and for $1\leq n \leq M$
	\begin{equation}\label{recursiva}
		|T(x)-T(y)|<\delta_{n-1}/2
	\end{equation}
provided that $|x-y|<\delta_{n}$. We take $\delta=\delta_M$ and let $\{x_0,x_1,x_2,\ldots\}$ be a $\delta$-pseudo orbit.
	
	Now, let $j\geq 0$ be a given index and we will prove by induction 
	\begin{equation}\label{iteraciones}
		|T^l(x_j)-x_{j+l}|<\delta_{M-l}
	\end{equation}
	for every $0\leq l\leq M$. First, we observe $|T(x_j)-x_{j+1}|<\delta<\delta_{M-1}$. Now, suppose that the property holds for some $1\leq l<M$ and we will prove it for $l+1$. Thus, \eqref{recursiva} yields 
	$$
	|T^{l+1}(x_j)-T(x_{j+l})|<\frac{\delta_{M-(l+1)}}{2}
	$$ 
	and hence 
	\begin{align*}
|T^{l+1}(x_j)-x_{j+l+1}|&\leq |T^{l+1}(x_j)-T(x_{j+l})|+|T(x_{j+l})-x_{j+l+1}|\\&<\frac{\delta_{M-(l+1)}}{2}+\delta<\delta_{M-(l+1)}.
	\end{align*}
	This proves \eqref{iteraciones} and the result follows immediately.
\end{proof}

\begin{lemma}\label{aroundFP}
Let $n\in \mathbb{N}$. Let $z_n^*=z_n-2^{-2n}$, where $z_n$ is the fixed point of the Takagi function defined in \eqref{z_n}. Then, $T\left(z_n^*\right)=z_n$, $T(x)>z_n$ for all $x\in \left(z_n^*,z_n\right)$ and 
$$
T^k\left(\left[z_n^*,z_n\right] \right)=[z_n, 2/3]
$$
for every $k\in\mathbb{N}$.
\end{lemma}
\begin{proof}
Since $G'^+_{2n}(z_n^*)=0$, it follows from \eqref{selfAffineProperty} that 
\begin{equation}\label{sap}
T\left (z_n^*+\frac{y}{2^{2n}}\right )=T(z_n^*)+\frac{1}{2^{2n}}T(y)
\end{equation}
for every $y\in [0,1]$. Hence, $T(z_n^*)=T(z_n)=z_n$ and $T(x)>z_n$ for all $x\in (z_n^*,z_n)$, see Figure \ref{Fig_zStar} with $z_n^*=z^*$ and $z_n=z$. Furthermore, the  graph of $T$ restricted to $[z^*_n,z_n]$ is a similar scaled down copy of the graph on the entire interval $[0,1]$. By applying \eqref{sap} with $y=2/3$, we obtain
$$
T\left (z_n^*+\frac{2}{3}\cdot \frac{1}{2^{2n}}\right )=\frac{2}{3}
$$
and consequently, $T\left ([z_n^*,z_n] \right )=[z_n,2/3]$. By  Theorem \ref{fixedPoints}, we have that $T(x)\geq x$ for all $x\in [z_n,2/3]$, and since $z_n$ and $2/3$ are both fixed points, we deduce that 
$$
[z_n,2/3]=T\left ([z_n,2/3] \right )=T^{k-1}\left ([z_n,2/3] \right )=T^{k-1}\left (T\left ([z_n^*,z_n] \right )\right )
$$ 
for every $k$. This proves the result.
\end{proof}

Let $\varepsilon>0$ be given. In Section \ref{BehaviorOrbitSection}, equation  \eqref{def_m},  we denoted by $m$ the index such that $z_m \in (2/3- \varepsilon / 2, 2/3)$ and $z_j\leq 2/3- \varepsilon / 2$ for all $j=0,\ldots,m-1$. Now, we also introduce the notation $\mathcal{Z}_\varepsilon = \{z_j: j=0,\ldots, m-1\}$.

For each $z\in \mathcal{Z}_\varepsilon$ we denote by $z^*$ the  preimage of $z$ given by Lemma \ref{aroundFP}, that is, $z^*$ is the largest preimage which is less than $z$ and $T(x)>z$ for all $x\in (z^*,z)$. By Lemma \ref{aroundFP} again, there is $\eta_z>0$ such that 
\begin{equation}\label{eqEta}
[z,z+\eta_z]\subsetneq [z,2/3]=T\left ([z^*,z]\right )= T^2\left ([z^*,z]\right ). 
\end{equation}
and  the interval $(z,z+\eta_z)$ contains no fixed points of the Takagi function (see Figure \ref{Fig_zStar}). 
\begin{figure}[h]
	\begin{center}
	\includegraphics[width=0.6\linewidth]{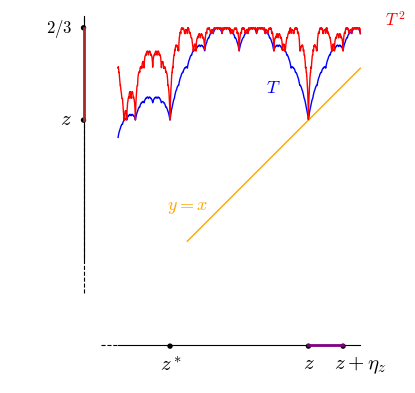}
	\caption{The graph of the $T$ and $T^2$ around $z\in \text{Fix}\,(T)$}
	\label{Fig_zStar}
\end{center}
\end{figure}

For all possible pairs $u,v\in \mathcal{Z}_\varepsilon$ with $u<v$ we apply Theorem \ref{Lemmam_0}  and we obtain the corresponding $\widetilde{\eta}(\varepsilon', u,v)$, and recall that for each $z\in \mathcal{Z}_\varepsilon$, $\eta_z$ is defined as in \eqref{eqEta}. There are finitely many choices for $u$ and $v$ in $\mathcal{Z}_\varepsilon$, hence we may choose $\eta_1>0$ such that 
\begin{equation}\label{eta1}
\eta_1= \min\left\{\, \min \{ \,\widetilde{\eta}(\varepsilon', u,v): u,v\in \mathcal{Z}_\varepsilon \text{ with } u<v\},  \min \{\eta_z: z\in \mathcal{Z}_\varepsilon\},\, \varepsilon'/4\right \}.
\end{equation}
Therefore,  if $\eta<\eta_1$ then Theorem \ref{Lemmam_0} applies to every pair $u,v\in \mathcal{Z}_\varepsilon$ with $u<v$ making a $(u,v,\eta)$-jump of a given order.

For each $0<\eta <\varepsilon'$ we define 
$$
N_{\max,\eta} = \frac{-\log\, \eta }{\log \, 100} +3m+ \frac{2/3}{\varepsilon'}m
$$
and observe that $N_{\eta}$ is the upper bound given by Theorem \ref{behavior}. The expression $\eta \cdot6^{N_{\max,\eta}}$ tends to 0 as $\eta$ goes to 0. Hence, there is $\eta_2>0$ such that  if $0<\eta<  \eta_2$ then we have
\begin{equation}\label{ControlSizePreimages}
\eta \cdot6^{2+N_{\max,\eta}}\leq \varepsilon/4.
\end{equation}
Now, we let
\begin{equation}\label{etaElection}
\eta= \frac{1}{2}\cdot \min \{\eta_1, \eta_2\}
\end{equation}
and
\begin{equation}\label{deltaElection}
 \text{ by invoking Lemma \ref{delta_eta} with $\rho = \eta/4$ and $M= N_{\max,\eta}+2$ we set $\delta>0$.}
\end{equation} 
Thus, if $\{x_0,x_1,x_2,\ldots\}$ is a $\delta$-pseudo orbit, then for any index $j\geq 0$ we have 
\begin{equation}\label{choiceOfDelta}
|T^l(x_j)-x_{j+l}|<\frac{\eta}{4}
\end{equation}
for every $l=0,1\ldots, N_{\max,\eta}+2$. Since $T(x)\leq 2/3$ for every $x\in [0,1]$, we have $T(x_k)\leq 2/3$ for every $k\geq 0$, and hence 
\begin{equation}\label{pseudoOrbit}
0\leq x_k\leq \frac{2}{3}+\frac{\eta}{4}
\end{equation}
for every $k\geq 1$.

\begin{proposition}\label{situationB2}
Let $\varepsilon>0$, and consider $m\in \mathbb{N}$ as defined in \eqref{def_m}, $\varepsilon'>0$ given by Lemma \ref{epsilonPrime}, $\eta>0$ given by \eqref{etaElection} and $\delta>0$ given by \eqref{deltaElection}. Let $u\in \mathcal{Z}_\varepsilon$ and $\{x_0,x_1,x_2,\ldots\}$ be a $\delta$-pseudo orbit. Assume that  there is  $j\in \mathbb{N}$ such that $x_j\in [u-\eta/4, u+\eta]$ and  $T^2(x_{j})> u+\eta$. Let $\overline{n}\in \mathbb{N}$ be such that $x_{j-l}\in [u-\eta/4, u+\eta]$ for every $l=0,\ldots, \overline{n}$. If there is $N\leq N_{\max,\eta}$ such that $T^{N-1}(x_j)< z_m$ and $T^N(x_j)\geq z_m$, then there exists ${\bf x_u}\in [u,u+\eta]$ such that 
$$
|T^{\, l}({\bf x_u})-x_{j-\overline{n}+l}|<\varepsilon
$$
for every $l\geq 0$.
\end{proposition}
\begin{proof}
Without loss of generality we may assume that $x_j\in [u,u+\eta]$, since otherwise we choose $\widetilde{x}_j\in [u,u+v]$ such that $T(x_j)=T(\widetilde{x}_j)$.  By Lemma \ref{InfinitelyPreimages} we can find a point ${\bf x_u}\in [u,z_j]$ such that 
$T^{n_u}({\bf x_u})=x_j$ and $u<T^{\, l}({\bf x_u})\leq x_j \leq u+\eta$ for every $l=0, \ldots, n_u$. Thus we have
$$
\left |T^{\, l}({\bf x_u})- x_{j-n_u+l}\right |<\frac{5}{4}\eta<\frac{5}{8}\varepsilon'< \frac{5}{16}\varepsilon
$$ 
for every $l=0, \ldots, n_u$, where we have used \eqref{etaElection}, \eqref{eta1} and Lemma \ref{epsilonPrime}. 

The choice of $\delta$ satisfying \eqref{choiceOfDelta}, together with $N\leq N_{\max}$ yields
$$
\left |T^{\, l}({\bf x_u})-x_{j-n_u+l}\right |<\frac{\eta}{4}<\frac{\varepsilon'}{8}<\frac{\varepsilon}{16}
$$
for every $l=n_u, \ldots, n_u+N+1$. 

Since $T^{n_u+N}({\bf x_u})\geq z_m$ we have that $x_{j+N}\geq z_m -\eta/4$ and it follows from \ref{away} in Lemma \ref{epsilonPrime} that $T(x_{j+N})\geq z_m$, which implies  $x_{j+N+1}\geq z_m -\eta/4$. It follows immediately that $x_{j+N+l}\geq z_m -\eta/4$ for every $l\geq 0$. Furthermore, we know that $T^{n_u+N+l}({\bf x_u})\geq z_m$ for every $l\geq 0$. Therefore, we conclude
$$
|T^{l}({\bf x_u})-x_{j-n_u+l}|< \frac{2}{3}+\frac{\eta}{4}-\left (z_m-\frac{\eta}{4} \right )=\frac{2}{3}-z_m+\frac{\eta}{2}<\frac{\varepsilon}{2}+\frac{\eta}{2}<\varepsilon
$$
for every $l\geq n_u+N+1$, where we have used \eqref{pseudoOrbit}.
\end{proof}

\begin{proposition}\label{situationB1}
Let $\varepsilon>0$, and consider $m\in \mathbb{N}$ as defined in \eqref{def_m}, $\varepsilon'>0$ given by Lemma \ref{epsilonPrime}, $\eta>0$ given by \eqref{etaElection} and $\delta>0$ given by \eqref{deltaElection}. Let $u\in \mathcal{Z}_\varepsilon$ and $\{x_0,x_1,x_2,\ldots\}$ be a $\delta$-pseudo orbit. Assume that there is $j\in \mathbb{N}$ satisfying that $x_j \in [u-\eta/4, u+\eta]$ and  $T^2(x_{j})> u+\eta$. Let $\overline{n}\in \mathbb{N}$ be such that $x_{j-l}\in [u-\eta/4, u+\eta]$ for every $l=0,\ldots, \overline{n}$. If there are $N\leq N_{\max}$ and $v\in \mathcal{Z}_\varepsilon$ with $u<v$ such that the point $x_j$ makes a $(u,v,\eta)$-jump of order $N$, then there exists a closed interval $W_u\subset  [u,u+\eta]$ such that
\begin{enumerate}[itemsep=0.3em,label=\small{(\roman*)}]
\item $T^{\,n_u+N}(W_u)=[v,v+\eta]$,
\item  $\lambda\left ( T^l\,(W_u)\right )<5\varepsilon/16$ for every $l=0,\ldots, \overline{n}+N$, and
 \item $\text{dist}\, (T^{\,l}(W_u), x_{j-n_u+l+1})<\frac{41}{64}\varepsilon
$ for every $l=0,\ldots, \overline{n}+N$.
\end{enumerate} 
Furthermore, we have $x_{j+N} \in [v-\eta/4, v+\eta]$.
\end{proposition}
\begin{proof}
As in Proposition \ref{situationB2}, we may assume $x_j\in [u,u+\eta]$. It follows from Lemma \ref{InfinitelyPreimages}  that we can find a point ${\bf x_u}\in [u, x_j]$ such that $T^{n_u}({\bf x_u})=x_j$ and $u<T^{\, l}({\bf x_u})\leq x_j \leq u+\eta$ for every $l=0, \ldots, n_u$. Thus, for every $l=0, \ldots, n_u$
$$
|T^{\, l}({\bf x_u})- x_{j-n_u+l}|<\frac{5}{4}\eta<\frac{5}{16}\varepsilon'< \frac{5}{64}\varepsilon
$$ 
 where we have used \eqref{eta1}, \eqref{etaElection} and Lemma \ref{epsilonPrime}. The choice of $\delta$ satisfying \eqref{choiceOfDelta}, together with $N\leq N_{\max}$ yields 
$$
|T^{\, l}({\bf x_u})-x_{j-n_u+l}|<\frac{\eta}{4}<\frac{\varepsilon'}{16}<\frac{\varepsilon}{64}
$$ 
for every $l=n_u+1, \ldots, n_u+N$, where we have used \eqref{eta1}, \eqref{etaElection}  and Lemma \ref{epsilonPrime} again. Therefore, we conclude for $l=0, \ldots, n_u+N$
\begin{equation}\label{x_u}
|T^{\, l}({\bf x_u})-x_{j-n_u+l}|<\frac{5}{64}\varepsilon.
\end{equation}

Recall that $v^*$ stands for the closest preimage of $v$ which is less than $v$, so we have $T(v^*)=v$ and by Lemma \ref{aroundFP} we know that $T(x)> v$ for every $x\in (v^*, v)$.  The point ${\bf x_u}$ makes a $(u,v,\eta)$-jump of order $n_u+N$ and it follows from Theorem \ref{Lemmam_0} that there exists $\widetilde{x}_u\in [u,T({\bf x_u})]$ that satisfies
\begin{equation}\label{x_tilde_pseudo}
T^{n_u+N}(\widetilde{x}_u) = v\quad \text{ and }\quad |T^{l}({\bf x_u})-T^l(\widetilde{x}_u)|<\varepsilon'
\end{equation}
for every $l=0, \ldots, n_u+N+1$. Now, observe that  
$$
T^{n_u+N-1}(\widetilde{x}_u)\leq v^*<v = T^{n_u+N}(\widetilde{x}_u) $$
and consequently, the function $T^{n_u+N-1}$ maps the interval $[\widetilde{x}_u, T(\widetilde{x}_u)]$ onto a closed interval  which contains $[v^*,v]$. Hence, Lemma \ref{aroundFP} yields that 
$$
T^{n_u+N}\left ( [\widetilde{x}_u, T(\widetilde{x}_u)]\right )=[v, 2/3]
$$
Let $v^*_0$ be the maximal element of $T^{-(n_u+N-1)} (v^*)$ in $[\widetilde{x}_u, T(\widetilde{x}_u)]$, and by \eqref{eta1} we obtain 
$$
[v,v+\eta]\subset [v,v+\eta_1]\subsetneq T([v^*,v])=[v, 2/3]= T^{n_u+N}\left ( [v^*_0, T(\widetilde{x}_u)]\right ).
$$
We observe that if $x\in T^{n_u+N-1}\left ( [v^*_0, T(\widetilde{x}_u)]\right )$ then $T(x)\geq v$. We denote by $y^*$ the maximal element of $T^{-(n_u+N)}(v+\eta)$ in $ [v^*_0, T(\widetilde{x}_u)]$ and we let $W_u=[y^*, T(\widetilde{x}_u)]$. Then, we have $T^{n_u+N}(W_u)=[v,v+\eta]$. 

Now, we recall that $T^l(\widetilde{x}_u)< T^{l+1}(\widetilde{x}_u)$ for every $l=0,\ldots, n_u+N-1$. We claim that $T^l(W_u)\subset [T^l(\widetilde{x}_u), T^{l+2}(\widetilde{x}_u) ]$ for every $l=0,\ldots, n_u+N-1$. 

Indeed, since otherwise, there exists an index $0\leq l_0\leq n_u+N-1$ such that either $T^{l_0}(\widetilde{x}_u)$ or  $T^{l_0+2}(\widetilde{x}_u)$ belongs to $T^{l_0}(W_u)$ since $T^{l_0}(W_u)$ is a closed interval and $T^{l_0+1}(\widetilde{x}_u)\in T^{l_0}(W_u)$. 

If  $T^{l_0}(\widetilde{x}_u)\in T^{l_0}(W_u)$ then $T^{n_u+N-1}(\widetilde{x}_u)$ and $T^{n_u+N}(\widetilde{x}_u)$ belong to $T^{n_u+N-1}(W_u)$, and as $T^{n_u+N-1}(\widetilde{x}_u)\leq v^*<v = T^{n_u+N}(\widetilde{x}_u) $, we obtain that $[v^*,v]\subset T^{n_u+N-1}(W_u)$ and by Lemma \ref{aroundFP} we obtain
$$
[v, 2/3]=T([v^*,v])\subset T^{n_u+N}(W_u)=[v,v+\eta],
$$
 which  contradicts \eqref{eqEta} and \eqref{eta1}. 

Otherwise, if $T^{l_0+2}(\widetilde{x}_u)$ belongs to $T^{l_0}(W_u)$, we have that $T^{n_u+N-1}(\widetilde{x}_u)$ and $T^{n_u+N}(\widetilde{x}_u)$ belong to $T^{n_u+N-2}(W_u)$, and we obtain that $[v^*,v]\subset T^{n_u+N-2}(W_u)$. Lemma \ref{aroundFP} again yields that
$$
[v, 2/3]=T^2([v^*,v])\subset T^{n_u+N}(W_u)=[v,v+\eta]
$$
which  contradicts \eqref{eqEta} and \eqref{eta1} again. 

Therefore, we conclude that 
\begin{equation}\label{InBetween}
T^l(W_u)\subset [T^l(\widetilde{x}_u), T^{l+2}(\widetilde{x}_u) ]
\end{equation}
for every $l=0,\ldots, n_u+N-1$.

Since $T^{n_u}({\bf x_u})=x_j\in [u,u+\eta]$ and $\widetilde{x}_u\in [u, T({\bf x_u})]$, it follows from \eqref{x_tilde_pseudo} that $|x_j-T^{n_u}(\widetilde{x}_u)|<\varepsilon'$ and hence, we have $T^{n_u}(\widetilde{x}_u)<u+\eta+\varepsilon'$. By \eqref{InBetween} we obtain
\begin{equation}\label{B1estimation}
\lambda\left (T^l(W_u) \right )<\eta+\varepsilon'<\frac{5}{4}\varepsilon'<\frac{5}{16}\varepsilon
\end{equation}
for every $l=0,\ldots,n_u-2$.

 On the other hand, Lemma \ref{expansive} together with \eqref{ControlSizePreimages} yields that 
$$
\lambda \left (T^{n_u+N-l}(W_u)\right )\leq \eta \cdot 6^{l}\leq  \eta\cdot 6^{2+N_{\max,\eta}} \leq \frac{\varepsilon}{4}
$$
for every $l=0,1,\ldots, N+2$. 

It follows from the previous results that $\lambda\left ( T^l\,(W_u)\right )<5\varepsilon/16$ for every $l=0,\ldots, n_u+N$.


Thus, if $x\in W_u$ then  
\begin{align*}
|T^{l}(x)- x_{j-n_u+l+1}|&\leq |T^{l}(x)- T^{l+1}(\widetilde{x}_u)|+ |T^{l+1}(\widetilde{x}_u)- T^{l+1}({\bf x_u})|\\&+|T^{l+1}({\bf x_u})-x_{j-n_u+l+1}|\\&\leq \frac{5}{16}\varepsilon+\varepsilon'+\frac{5}{64}\varepsilon<\frac{41}{64}\varepsilon
\end{align*}
for every $l=0,\ldots, n_u+N$, where we have used that $T(\widetilde{x}_u)\in W_u$ together with \eqref{x_u}, \eqref{x_tilde_pseudo} and \eqref{B1estimation}.

Finally, we must prove that $x_{j+N} \in [v-\eta, v+\eta]$. 

We assume first that  $v-\eta\leq T^N(x_j)\leq v$. 

It follows from \eqref{choiceOfDelta} that $v-\frac{5}{4}\eta<x_{j+N}< v+\frac{\eta}{4}$. 

Observe that if $v-\frac{5}{4}\eta<x_{j+N}<v-\eta$, then $T(x_{j+N})\geq v+100^m\eta$ by \ref{away} in Lemma \ref{epsilonPrime}, which implies 
$$
x_{j+N+1}\geq  v+100^m\eta -  \frac{\eta}{4}\geq v+99\eta
$$
and hence $T^{N+1}(x_j)\geq v+99\eta-\frac{\eta}{4}\geq v+98\eta$. This contradicts that $T^{N+1}(x_j) \in [v-\eta, v+\eta]$. 

Therefore we conclude that $v-\eta \leq x_{j+N}<v+\frac{\eta}{4}$.

Now,  assume  that  $v< T^N(x_j)\leq v+\eta$. By \eqref{choiceOfDelta} again, we have $v-\frac{\eta}{4}<x_{j+N}<v+\frac{5}{4}\eta$. 

If $v+\eta<x_{j+N}<v+\frac{5}{4}\eta$ then $T(x_{j+N})\geq v+100^m\eta$ and the argument proceeds as before. Hence, we deduce that $v-\frac{\eta}{4}<x_{j+N}<v+\eta$.
\end{proof}

\begin{thm}\label{main}
The Takagi function has the shadowing property.
\end{thm}
\begin{proof}
Let $0<\varepsilon<2/3$ be given. Let $m\in \mathbb{N}$ be defined as in \eqref{def_m} and let $\varepsilon'>0$ be given by Lemma \ref{epsilonPrime}. Let $\eta>0$ be chosen as in \eqref{etaElection} and let $\delta>0$ be obtained by invoking Lemma \ref{delta_eta} with $\rho = \eta/4$ and $M= N_{\max,\eta}+2$. 

 Let $\{x_0,x_1,x_2,\ldots\}$ be a $\delta$-pseudo orbit. We must show that there is $x^*\in [0,1]$ such that
\begin{equation}\label{resultWanted}
|T^l(x^*)-x_l|<\varepsilon
\end{equation}
for every $l\geq 0$.

We may assume without loss of generality that $x_0\in [0,2/3]$. Indeed, if $x_0\in (2/3,1]$, then, by the symmetry of the Takagi function about the line $x=1/2$, there exists $\overline{x}_0\in [0,1/2)$ such that $T(\overline{x}_0)=T(x_0)$. Now, we consider the $\delta$-pseudo orbit given by $\left \{\overline{x}_0,x_1,x_2,\ldots\right \}$ and suppose we have a point $\overline{x}^*$ such that $|\overline{x}^*-\overline{x}_0|<\varepsilon$ and $|T^l(\overline{x}^*)-x_l|<\varepsilon$ for every $l\geq 1$. By the symmetry of the Takagi function about the line $x=1/2$, we can find a point $x^*$ such that $T(\overline{x}^*)=T(x^*)$ and $|\overline{x}^*-x^*|<\varepsilon$. Thus, we have $|T^l(x^*)-x_l|<\varepsilon$ for every $l\geq 0$.

By virtue of Lemma \ref{InfinitelyPreimages} and by adding some extra points to the pseudo-orbit if necessary, we may also assume that $x_0\in [u_0, u_0+\eta]$ for some $u_0\in \text{Fix}(T)$. 

It follows from \eqref{pseudoOrbit} and \ref{away} in Lemma \ref{epsilonPrime} that 
$$
x_l\in \left ( u_0-\frac{\eta}{4}, \frac{2}{3}+\frac{\eta}{4}\right ) 
$$
for every $l\geq 0$.

Furthermore, observe that if $u_0\geq z_m$, then 
$$
x_l\in \left ( z_m-\frac{\eta}{4}, \frac{2}{3}+\frac{\eta}{4}\right ) 
$$
for every $l\geq 0$, and 
$$
u_0\in \left [z_m, \frac{2}{3}\right ]\subset \left ( z_m-\frac{\eta}{4}, \frac{2}{3}+\frac{\eta}{4}\right ) 
$$
which implies by \eqref{def_m}, \eqref{eta1}, \eqref{etaElection}, and Lemma \ref{epsilonPrime}
$$
|T^l(u_0)-x_l|=|u_0-x_l|<\frac{2}{3}+\frac{\eta}{4} - \left ( z_m-\frac{\eta}{4}\right )<\frac{\varepsilon}{2}+\frac{\eta}{2}<\frac{\varepsilon}{2}+\frac{\varepsilon'}{8}<\varepsilon
$$
for every $l=0,1,2,\ldots$This gives the result \eqref{resultWanted} where $x^*=u_0$.

Therefore, we may assume without loss of generality that $u_0<z_m$ and $x_0\in [u_0,u_0+\eta]$.

By induction we define a finite sequence of indices $n_0,n_1,\ldots, n_{p}$ and a finite sequence of closed, possibly degenerate, intervals $W_0,\ldots, W_p$ for some $0\leq p \leq m-1$. We take $n_0=0$.

If  $x_i\in [u_0-\eta, u_0+\eta]$ for all $i\geq 0$, then we take $W_0=\{u_0\}$ and we let $p=0$.

Otherwise, there exists an index $j_0\geq 0$ such that $x_{i}\in [u_0-\eta, u_0+\eta]$ for all $i=0,\ldots, j_0$ and $x_{j_0+1}>u_0+\eta$. Without loss of generality, we may assume that $x_{j_0}\in [u_0,u_0+\eta]$. Hence, we have $T^2(x_{j_0})>u_0+\eta$. In view of Theorem  \ref{behavior}, there exists $N_0\leq N_{\max,\eta}$ such that exactly one of the following holds:

 If $T^{N_0-1}(x_{j_0})< z_m$ and $T^{N_0}(x_{j_0})\geq z_m$, that is, case \ref{case1} of Proposition \ref{behavior} applies,  then we invoke Proposition \ref{situationB2} with $j=j_0$ and $\overline{n}=j_0$, and we obtain ${\bf x_{0}}\in [u_0, u_0+\eta]$ such that 
$$
|T^l({\bf x_{0}})-x_{l}|<\varepsilon
$$
for every $l=0,1,2,\ldots$ We let $W_0=\{{\bf x_{0}}\}$ and $p=0$.

Otherwise, \ref{case2} of Proposition \ref{behavior} applies, that is, there exists $u_1\in \mathcal{Z}_\varepsilon$ with $u_1>u_0$ such that $x_{j_1}$ makes a $(u_0,u_1, \eta)$-jump of order $N_0$. We invoke Proposition \ref{situationB1} with $j=j_0$ and $\overline{n}=j_0$, and we obtain a closed interval $W_0\subset [u_0,u_0+\eta]$ such that
$$
T^{j_0+N_0}(W_0)=[u_1,u_1+\eta]
$$
and if $x\in W_0$ then 
$$
|T^l(x)-x_{l+1}|<\frac{41\varepsilon}{64}
$$
for every $l=0,\ldots, j_0+N_0-1$. Furthermore, we have $x_{j_0+N_0} \in [u_1-\eta, u_1+\eta]$.  

Now, suppose that $n_k$ and $W_k$ has been defined for some $k\geq 0$ satisfying that $W_{k}\subset [u_k,u_k+\eta]$ together with
$$
T^{j_k+N_k}\left ( W_k\right )= [u_{k+1},u_{k+1}+\eta]
$$
for some $u_{k+1}<z_m$, and if $x\in W_k$ then 
$$
|T^l(x)-x_{n_k+l+1}|<\frac{41\varepsilon}{64}
$$
for every $l=0,\ldots, j_k+N_k-1$. Moreover, we have $x_{n_k+j_k+N_k}\in [u_{k+1}-\eta, u_{k+1}+\eta]$. 

We let $n_{k+1}=n_k+j_k+N_k$.

If $x_{n_{k+1}+i}\in [u_{k+1}-\eta/4, u_{k+1}+\eta]$ for all $i\geq 0$, then we take $W_{k+1}=\{u_{k+1}\}$ and $p=k+1$.

Otherwise, there exists $j_{k+1}\geq n_{k+1}$ such that $x_{i}\in [u_{k+1}-\eta, u_{k+1}+\eta]$ for all $i=0,\ldots, j_{k+1}$ and $x_{j_{k+1}+1}>u_{k+1}+\eta$. Without loss of generality, we may assume that $x_{j_{k+1}}\in [u_{k+1},u_{k+1}+\eta]$. Hence, we have $T^2(x_{j_{k+1}})>u_{k+1}+\eta$. As before, it follows from Theorem  \ref{behavior}  that there exists $N_{k+1}\leq N_{\max,\eta}$ such that exactly one of the following holds:

 If $T^{N_{k+1}-1}(x_{j_{k+1}})< z_m$ and $T^{N_{k+1}}(x_{j_{k+1}})\geq z_m$, that is, \ref{case1} of Proposition \ref{behavior} applies,  then we invoke Proposition \ref{situationB2} with $j=j_{k+1}$ and $\overline{n}=j_{k+1}-n_{k+1}$, and we obtain ${\bf x_{k+1}}\in [u_{k+1}, u_{k+1}+\eta]$ such that 
$$
|T^l({\bf x_{{k+1}}})-x_{n_{k+1}+l+1}|<\varepsilon
$$
for every $l=0,1,2,\ldots$ We let $W_{k+1}=\{{\bf x_{{k+1}}}\}$ and $p={k+1}$.

Otherwise, \ref{case2} of Proposition \ref{behavior} applies, that is, there is $u_{k+2}<z_m$ with $u_{k+1}<u_{k+2}$ such that $x_{j_{k+1}}$ makes a $(u_{k+1},u_{k+2}, \eta)$-jump of order $N_{k+1}$. We invoke Proposition \ref{situationB1} with $j=j_{k+1}$ and $\overline{n}=j_{k+1}-n_{k+1}$, and we obtain a closed interval $W_{k+1}\subset [u_{k+1},u_{k+1}+\eta]$ such that
$$
T^{j_{k+1}+N_{k+1}}(W_{k+1})=[u_{k+2},u_{k+2}+\eta],
$$
and if $x\in W_{k+1}$ then 
$$
|T^l(x)-x_{n_{k+1}+l+1}|<\frac{41\varepsilon}{64}
$$
for every $l=0,\ldots, j_{k+1}+N_{k+1}$.  Furthermore, we have 
$$
x_{n_{k+1}+ j_{k+1}+N_{k+1}} \in [u_1-\eta/4, u_1+\eta].
$$

This process can be repeated $p$ times, with $p\leq m-1$. As a consequence, we obtain a finite sequence of closed intervals $W_0,\ldots, W_p$ where $W_{p}$ consists of a single point ${\bf x_{p}}\in [u_p, u_p+\eta]$ such that 
\begin{equation}\label{f1}
|T^l({\bf x_{p}})-x_{n_{p}+l+1}|<\varepsilon
\end{equation}
for every $l=0,1,2,\ldots$ Furthermore, for every $k=0,\ldots, p-1$ we have
\begin{equation}\label{key}
T^{\,j_k+N_k}(W_k)=[u_{k+1},u_{k+1}+\eta]
\end{equation}
and if $x\in W_{k}$ then 
$$
|T^l(x)-x_{n_{k}+l+1}|<\frac{41\varepsilon}{64}
$$
for every $l=0,\ldots, j_k+N_k-1$.

It follows from \eqref{key} that there is a sequence of points $y_0,\ldots, y_p$ with $y_p= {\bf x_{p}}$ such that for every $k=0,\ldots, p-1$ we have $y_k\in W_k$, $T^{j_k+N_k}(y_k)=y_{k+1}$ and 
\begin{equation}\label{f2}
\left |T^l(y_k)-x_{n_k+l+1}\right | < \frac{41\varepsilon}{64}
\end{equation}
for every $l=0,\ldots, j_k+N_k-1$. By Lemma \ref{InfinitelyPreimages} we can choose $x^*\in [u_0,u_0+\eta]$ such that $T(x^*)=y_0$. This, along with \eqref{f1} and \eqref{f2}, yields
$$
\left |T^l(x^*)-x_l\right |<\frac{41\varepsilon}{64}<\varepsilon
$$
for every $l\geq 0$. This gives the result.
\end{proof}

\section{The shadowing property and the Takagi family}\label{NoShadowingSection}
For parameters $\gamma>0$ the Takagi family consists of all the maps 
$$
\textbf{T}_\gamma:\mathbb{R}\to \bigg[0,\frac{2}{3}\cdot \gamma\bigg] \text{ defined by }\textbf{T}_\gamma=\gamma \cdot T
$$
 Here, we use the fact that the Takagi function can be extended periodically to the entire real line with period one.

A substantial part of this paper was devoted to establishing that the Takagi function ($\gamma=1$) has the shadowing property. Nevertheless, this phenomenon does not persist throughout the Takagi family. In this section we determine values of the parameter $\gamma$ for which the function $\mathbf{T}_\gamma$ does not have the shadowing property. This is done by studying two distinct geometric properties satisfied by the Takagi function at certain points, thereby revealing new geometric features of the function.

The first such property is the notion of a quasi-tangent point for a function.
\begin{defn}\label{tangent}
We say that $x\in \mathbb{R}$ is a quasi-tangent point for  a function $f:\mathbb{R}\to \mathbb{R}$ if there exist $\xi>0$ and $r>0$ such that 
\begin{enumerate}[label=\small{(\roman*)}]
	\item\label{t1} $f(x)=\xi x$,
	\item\label{t2} $f(y)\leq \xi y$ for every $y\in (x-r,x+r)$, and
	\item\label{t3} $f(y)>\xi x$ for every $y\in (x,x+r)$.
\end{enumerate}
\end{defn}

The set of points at which the Takagi function has a quasi-tangent point has been characterized in \cite{BLL} in terms of the binary expansion of those points. Figure \ref{fig:tangent_points} below illustrates the graph of $\textbf{T}_{\gamma}$, where $\gamma=13/15$, near the point $x=13/24$, which is a quasi-tangent point for the Takagi function with $\xi=15/13$.

\begin{proposition}
If $x\in [0,1]$ is a quasi-tangent point of the Takagi function, then there exists $\gamma = \gamma(x)>0$ such that $\textbf{T}_{\gamma}$ does not have the shadowing property.
\end{proposition}
\begin{proof}
Let $x\in [0,1]$ be a quasi-tangent point for the Takagi function. Thus, there exist $\xi>0$ and $r>0$ such that  \ref{t1}, \ref{t2} and \ref{t3} in Definition \ref{tangent} hold. 

We let $\gamma = 1/\xi$ and we have that $\textbf{T}_{\gamma}(x)= x$,  $\textbf{T}_{\gamma}(y)\leq y$ for every $y\in (x-r,x+r)$, and $\textbf{T}_{\gamma}(y)> x$ for every $y\in (x,x+r)$.

Let $\varepsilon'=r/4$. For the sake of contradiction, suppose that there exists $\delta>0$ such that for every $\delta$-pseudo orbit $(x_j)_j$ there is $x^*\in [0,1]$ satisfying
\begin{equation}\label{sake2}
\left |T^j(x^*)-x_j \right|<\varepsilon'
\end{equation}
for every $j=0,1,\ldots$ Without loss of generality we may assume that $\delta<\varepsilon'$. We will construct a $\delta$-pseudo orbit $\left \{x_0,x_1,x_2,\ldots\right \}$ for which \eqref{sake2} does not hold.

Let $x_0\in\left (x+r/2,x+3r/4\right )$. Since $x<\textbf{T}_{\gamma}(y)\leq y$ for every $y\in (x,x+r)$, we have that $\left ( \textbf{T}_{\gamma}^n(x_0)\right )_n$ is a decreasing sequence converging to $x$. Thus, there exists $n_0\in \mathbb{N}$ such that $\textbf{T}_{\gamma}^n(x_0)\in (x, x+\delta/2)$ for every $n\geq n_0$. We take $x_n=\textbf{T}_{\gamma}^n(x_0)$ for every $n=1,\ldots,n_0$.

Now, we take $x_{n_0+1}= \textbf{T}_{\gamma}^{n_0}(x_0)-\delta/2=x_{n_0}-\delta/2$, and we have that $x_{n_0+1}\in (x-r/4,x)$.  Since $\textbf{T}_{\gamma}(y)\leq y$ for every $y\in (x-r,x)$, we have that $\left (T^n (x_{m+1})\right )_n$ is  a decreasing sequence, and there exists $l_0\in \mathbb{N}$ such that 
$$
\textbf{T}_{\gamma}^{l_0} (x_{n_0+1})<x-r/2 \quad \text{ and }\quad \textbf{T}_{\gamma}^{l_0-1} (x_{n_0+1})\geq x-r/2.
$$
We take $$x_n= \textbf{T}_{\gamma}^{n-n_0-1}(x_{n_0+1})$$ for every $n>n_0+1$. Therefore, we have that $x_{l_0+n_0+1}<x-r/2$.

Suppose there exists $x^*\in [0,1]$ such that \eqref{sake2} holds for the $\delta$-pseudo orbit that we have constructed. Since $x_0\in\left (x+r/2,x+3r/4\right )$ and $\varepsilon'=r/4$ we obtain that $x+r/4<x^*<x+r$, and hence $\left ( \textbf{T}_{\gamma}^n(x^*)\right )_n$ is a decreasing sequence converging to $x$. Furthermore, $\textbf{T}_{\gamma}^n(x^*)\geq x$ for every $n$ (see Figure \ref{fig:tangent_points}). However, since $x_{l_0+n_0+1}<x-r/2$ we have 
$$
\textbf{T}_{\gamma}^{l_0+n_0+1}(x^*)-x_{l_0+n_0+1}>\frac{r}{2}>\varepsilon'
$$
which contradicts \eqref{sake2}.
\begin{figure}[h]
    \centering
    \includegraphics[width=0.95\linewidth]{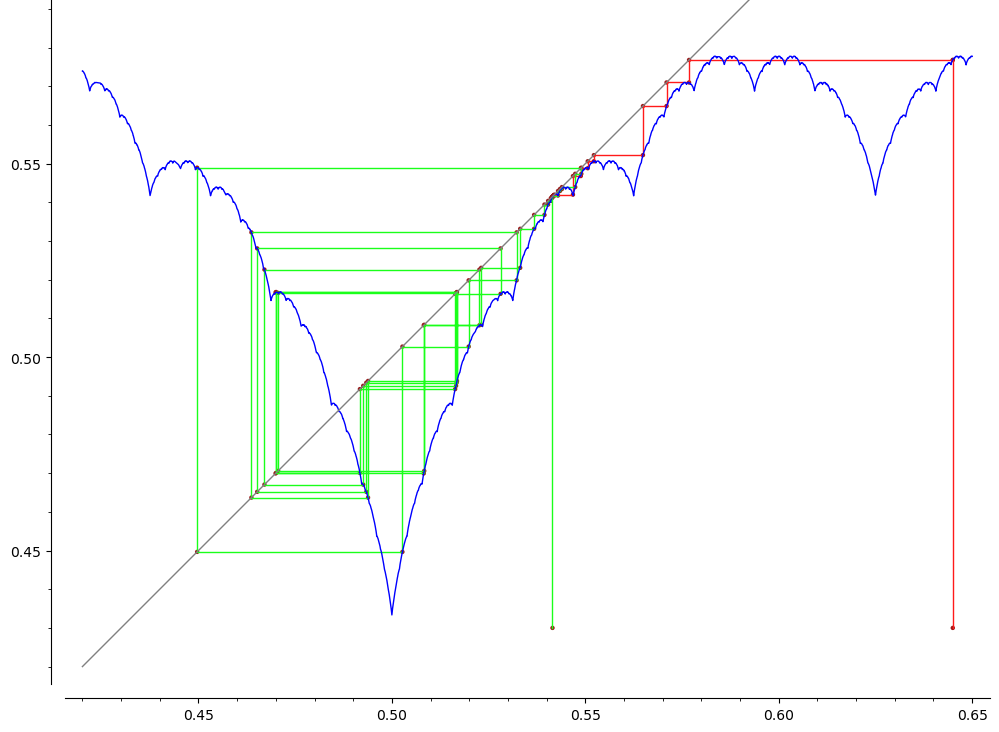}
    \caption{The quasi-tangent point $x=13/24\approx0.541666667$ and $\textbf{T}_{\gamma}$, where $\gamma=\gamma(x)=13/15$.}
    \label{fig:tangent_points}
\end{figure}
\end{proof}

We now turn to the study of the second geometric property. We use the following two results: the first is due to P. C. Allaart (see Lemma 3.3 of \cite{A1}), and the second was obtained recently by the second author and B. Hanson (see Lemma 3.5 of \cite{HLL}).

\begin{lemma}\label{Allaart}
	Let $x\in D_{2k+1}$ with binary expansion, $x=0.\varepsilon_1\varepsilon_2\varepsilon_3\ldots$ be such that $G'^+_{2k}(x)=0$ and $\varepsilon_{2k}(x)=1$. For each non-negative integer $n$ we define
	$$
	a_n = x-\sum_{j=1}^n \frac{1}{4^{k+j}}.
	$$
	Then, we have 
	$$
	T\left ( a_{n+1}+ \frac{y}{4^{k+n+1}}\right )=T(x)+\frac{1}{4^{k+n+1}}T\left ( y\right )
	$$
	for all $y\in [0,1]$.
\end{lemma}

\begin{lemma}\label{line}
For every $x\in \left[0,1/6\right]$ we have
$$
T(x)\leq \frac{1}{2}-2\left ( \frac{1}{6}-x\right ).
$$
\end{lemma}

For every $y\in [0,2/3]$, the level set of the Takagi function at $y$ is defined by 
$$
L(y)=\left \{ x\in [0,1]: T(x)=y\right \}.
$$
In 2012, P. C. Allaart proved that if $y$ is a dyadic rational other than zero, then the set $L(y)$ is countable (see Proposition 4.5 in \cite{A2}). Therefore, for every $n\in \mathbb{N}$ we have that $L(z_n)$ is countable. Recall that $z_n$ is the fixed point of $T$ defined in \eqref{z_n}. This second geometric property is closely related to the structure of the level set of the Takagi function at  $z_n$ on the interval $[z_n, z_n + 2^{-(2n-1)}]$. In particular, the largest element  of $L(z_n)$ contained in this interval plays a crucial role. For simplicity, we present this property in the case of the level set at the point $z_1=1/2$, where the largest element of $L(1/2)$ whithin $[1/2,1]$ is $5/6$.

\begin{lemma}\label{L1/2}
For each non-negative integer $n$ we define
	$$
	y_n= \frac{3}{4}+\frac{1}{4}\sum_{j=1}^{n}\frac{1}{4^k}.
	$$
Then, 
\begin{equation}\label{smallHumps}
T\left ( y_n+\frac{x}{4^{n+2}}\right )=\frac{1}{2}+\frac{1}{4^{n+2}}T(x)
\end{equation}
for all $x\in [0,1]$. Furthermore, we have the following:
\begin{enumerate}[label=\small{(\roman*)}]
\item\label{reverseLine} If $x\in \left[5/6,1\right]$, then
$$
T(x)\leq \frac{1}{2}-2\left ( x-\frac{5}{6}\right ).
$$
\item\label{levelSetRight}  $L\left(1/2\right)\cap \left[1/2,1\right]=\left \{1/2\right\}\cup \left \{y_n:n\geq 0\right \}\cup \left \{5/6\right\}$.
\end{enumerate}
\end{lemma}
\begin{proof}
In order to prove the result, we will first make use of Lemma \ref{Allaart} to study the Takagi function on the interval $\left[0, 1/2\right]$, and then take advantage of its symmetry about the line $x = 1/2$.

By virtue of the self-affine property \eqref{selfAffineProperty}, we have $T(1/4) = T(1/2) = 1/2$, and moreover $T(x)>1/2$ provided that $x\in (1/4, 1/2)$.

To apply Lemma \ref{Allaart} to $x=1/4$, and consequently $k=1$, we define
	$$
	a_n = \frac{1}{4}-\frac{1}{4}\sum_{j=1}^n \frac{1}{4^{j}}=\frac{1}{6}+\frac{1}{3}\cdot \frac{1}{4^{n+1}}.
	$$
for every non-negative integer $n$. Observe that $a_0=1/4$. It follows from Lemma \ref{Allaart} that 
\begin{equation}\label{smallHumpsa_n}
T\left ( a_{n+1}+ \frac{y}{4^{n+2}}\right )=\frac{1}{2}+\frac{1}{4^{n+2}}T\left ( y\right )
\end{equation}
for all $y\in [0,1]$. Therefore, for every $n\geq 0$ we have $T(a_n)=1/2$ and $T(x)>1/2$ provided that $x\in (a_{n+1}, a_n)$. Since $(a_n)_n$ converges to $1/6$ we obtain $T(1/6)=1/2$. By Lemma \ref{line} we conclude that 
$$
L\left(1/2\right)\cap \left[0,1/2\right]=\left \{1/6\right\}\cup \left \{a_n:n\geq 0\right \}\cup \left \{1/2\right\}.
$$

Now, observe that for every $n \geq 0$ we have 
$$
y_n=1-a_n= \frac{3}{4}+\frac{1}{4}\sum_{j=1}^{n}\frac{1}{4^j}=\frac{5}{6}-\frac{1}{3}\cdot \frac{1}{4^{n+1}}
$$
and by using the symmetry of the Takagi function about the line $x = 1/2$, together with \eqref{smallHumpsa_n}, we obtain 
\begin{align*}
T\left ( y_n+\frac{x}{4^{n+2}}\right )&=T\left ( a_n-\frac{x}{4^{n+2}}\right )=T\left ( a_{n+1}+\frac{1-x}{4^{n+2}}\right )=\frac{1}{2}+\frac{1}{4^{n+2}}T(1-x)\\&=\frac{1}{2}+\frac{1}{4^{n+2}}T(x)
\end{align*}
for all $x\in [0,1]$. This proves \eqref{smallHumps}.

Finally, if $x \in [5/6,1]$, then $1-x \in [0,1/6]$, and by applying Lemma \ref{line} to $1-x$ we obtain \ref{reverseLine}. Moreover, \ref{levelSetRight} follows immediately from the previous results.
\end{proof}

\begin{proposition}\label{5/3}
If $\gamma = 5/3$, then $\textbf{T}_{5/3}$ does not have the shadowing property.
\end{proposition}
\begin{proof}
Following the notation used in Lemma \ref{L1/2}, for each non-negative integer $n$ we define
\begin{equation}\label{y_n}
	y_n= \frac{3}{4}+\frac{1}{4}\sum_{j=1}^{n}\frac{1}{4^j}=\frac{5}{6}-\frac{1}{3}\cdot \frac{1}{4^{n+1}}.
\end{equation}
By Lemma \ref{L1/2} we obtain
$$
\textbf{T}_{5/3}\left (\frac56\right )=\frac{5}{3}\cdot T\left (\frac56\right )=\textbf{T}_{5/3}\left (y_n\right )=\textbf{T}_{5/3}\left (\frac12\right )=\frac{5}{6}
$$
Observe that the sequence $(y_n)_n$ is strictly increasing and converges to $5/6$. Furthermore, we have $y_{n+1}=y_n+4^{-(n+2)}$.

For every non-negative integer $n$ we consider the interval $$I_n=\left [y_{n}+\frac{1}{2}\cdot \frac{1}{4^{n+2}}, y_{n+1}\right ]$$ We will prove that 
\begin{equation}\label{ContainingPreviousHump}
I_{n-1}=\left [y_{n-1}+\frac{1}{2}\cdot\frac{1}{4^{n+1}}, y_n\right ]\subset \textbf{T}_{5/3}^{\, 2}\left (I_n \right )
\end{equation}
for every $n\in \mathbb{N}$ (see Figure \ref{Noshadowing5/3}). 

Indeed, it follows from \eqref{smallHumps} in Lemma \ref{L1/2} that
\begin{equation}\label{ImageI_n}
\textbf{T}_{5/3}\left ( I_n\right )=\frac{5}{3}\cdot T\left (I_n\right )=\left [\frac{5}{6}, \frac{5}{6}+\frac{2}{3}\cdot\frac{5}{3} \cdot \frac{1}{4^{n+2}} \right ]=\left [\frac{5}{6}, \frac{5}{6}+\frac{10}{9}\cdot \frac{1}{4^{n+2}} \right ].
\end{equation}
For every $x\in [5/6,1]$ we define the line 
$$
\mathfrak{T}(x)=\frac{5}{3}\cdot \frac{1}{2}-2\cdot \frac{5}{3} \cdot\left ( x-\frac{5}{6}\right)=\frac{5}{6}-\frac{10}{3}\left ( x-\frac{5}{6}\right)
$$
and by \ref{reverseLine} in Lemma \ref{L1/2} we obtain 
\begin{equation}\label{less}
\textbf{T}_{5/3}\left (\frac{5}{6}+\frac{10}{9}\cdot \frac{1}{4^{n+2}} \right )\leq\mathfrak{T}\left (\frac{5}{6}+\frac{10}{9}\cdot \frac{1}{4^{n+2}} \right ) = \frac{5}{6}-\frac{100}{27}\cdot \frac{1}{4^{n+2}}.
\end{equation}
Since  $\textbf{T}_{5/3}\left (5/6\right)=5/6$, the continuity of $\textbf{T}_{5/3}$ together with \eqref{ImageI_n} and \eqref{less} yields that
$$
\left [ \frac{5}{6}-\frac{100}{27}\cdot \frac{1}{4^{n+2}}, \frac{5}{6}\right ]\subset \textbf{T}_{\gamma}^{\, 2}\left (I_n \right ).
$$
By \eqref{y_n} we have
$$
 \frac{5}{6}-\frac{100}{27}\cdot \frac{1}{4^{n+2}}\leq y_{n-1}+\frac{1}{2}\cdot \frac{1}{4^{n+1}}<y_n<\frac{5}{6}
$$
and this gives the desired result.

Let 
$$
\varepsilon'=\frac{5}{6}-y_2=\frac{1}{3}\cdot \frac{1}{4^3}.
$$
 For the sake of contradiction, suppose that there exists $\delta>0$ such that for any $\delta$-pseudo orbit $(x_j)_j$ there is $x^*\in [0,1]$ such that 
\begin{equation}\label{sake}
\left |T^j(x^*)-x_j \right|<\varepsilon'
\end{equation}
for every $j=0,1,\ldots$  

We will construct a $\delta$-pseudo orbit $\left \{x_0,x_1,x_2,\ldots\right \}$ for which there is no $x^*\in [0,1]$ satisfying \eqref{sake}.

Since $\textbf{T}_{5/3}(1/2)=5/6$, we let $x_0=1/2$. Since $(y_n)_n$ is a stricly increasing sequence converging to $5/6$, we choose $m\in\mathbb{N}$ such that 
\begin{equation}\label{p_n}
I_m\subset \left [\frac{5}{6}-\frac{\delta}{2}, \frac{5}{6} \right ]\quad  \text{ and }\quad I_{m-1}\not\subset \left [\frac{5}{6}-\frac{\delta}{2}, \frac{5}{6} \right ].
\end{equation}

Now, we inductively define a finite strictly increasing sequence of points $p_0,\ldots, p_m$ such that $p_n\in \inte\left (I_n\right )$ for every $n=0,\ldots, m$ and $\textbf{T}_{5/3}^2(p_n)=p_{n-1}$ for every $n=1,\ldots, m$. We let
$$
p_0= y_{0}+\frac{2}{3}\cdot \frac{1}{4^{2}}\in \inte\left (I_0\right ).
$$
Suppose that $p_n\in \inte\left(I_n\right)$ has been defined for some $n=0,\ldots, m-1$. Since $\textbf{T}_{5/3}^{2}(p_n)\neq 5/6$, it follows from \eqref{ContainingPreviousHump} that there exists $p_{n+1}\in \inte\left(I_{n+1}\right)$ such that $\textbf{T}_{5/3}^{2}(p_{n+1})=p_n$. 

We take $x_1=p_m$. Since $p_m\in I_m$, by \eqref{p_n} we obtain that
$$
\left |x_1-\textbf{T}_{5/3}(x_0)\right|=\left |x_1-\frac{5}{6}\right|<\frac{\delta}{2}.
$$
For every $n\geq 2$ we take $x_n= \textbf{T}_{5/3}^{\, n-1}(x_1)$. Therefore, we have 
$$
\textbf{T}_{5/3}^{\, 2k}(x_1)=p_{m-k}\in \inte\left(I_{n-k}\right)
$$
for every $k=0,\ldots, m$.

Let $x^*\in (5/6-\varepsilon', 5/6+\varepsilon')$ such that \eqref{sake} holds. Since $$
|x^*-x_0|=\left |x^*-\frac{1}{2}\right |<\varepsilon'<\frac{1}{8}
$$ 
we have that  $x^*\in (3/8, 5/8)$. Moreover, $x_{2n+1}=\textbf{T}_{5/3}^{\, 2n}(x_1)\in I_0$ and hence 
$$
x_{2n+1}<\frac{5}{6}-\varepsilon'
$$
which implies $\textbf{T}_{5/3}^{\, 2n+1}(x^*)<\frac{5}{6}$. However, since $x^*\in (1/4, 3/4)$ it follows that $\textbf{T}_{5/3}^{2j+1}(x^*)\geq 5/6$ for every $j=0,\ldots, n$. This is a contradiction.

\begin{figure}[h]
	\centering
	\includegraphics[width=\linewidth]{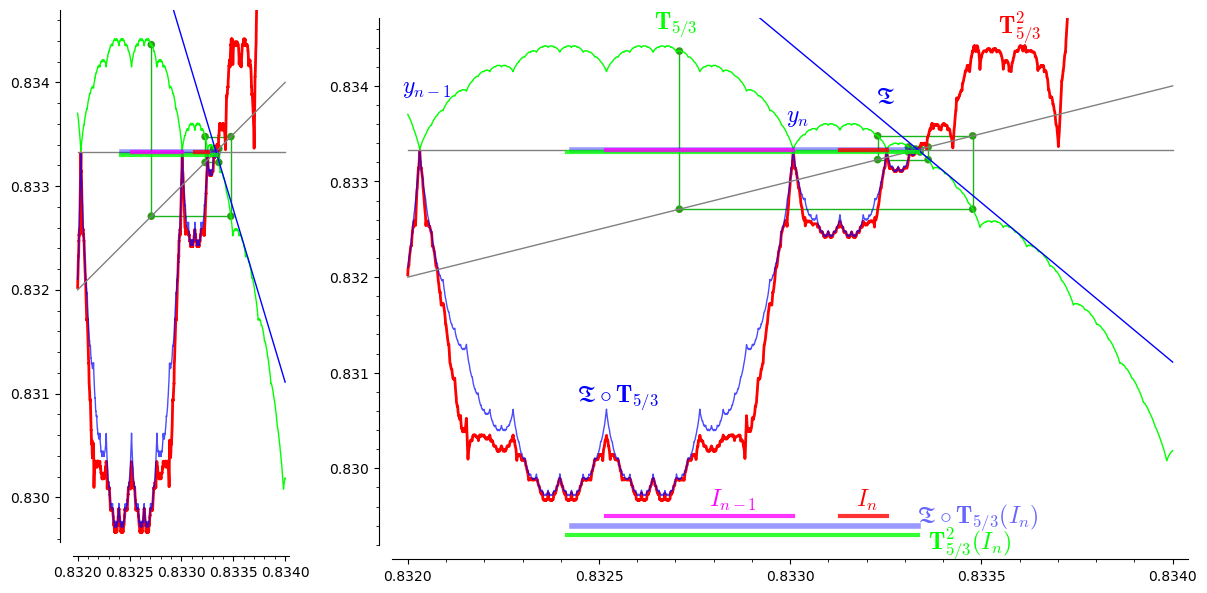}
	\caption{Two distorted images of parts of the graphs of $\textbf{T}_{5/3}$ and $\textbf{T}_{5/3}^{\, 2}$ with the line $\mathfrak{T}$. The interval $I_{n-1}=\left [y_{n-1}+\frac{1}{2}\cdot\frac{1}{4^{n+1}}, y_n\right ]$ satisfies $I_{n-1}\subset \textbf{T}_{5/3}^{\, 2}\left (I_n \right )$.}
	\label{Noshadowing5/3}
\end{figure}
\end{proof}
It is important to observe that the value  $\gamma = 5/3$ appearing in Proposition \ref{5/3} satisfies
$$
\textbf{T}_{\gamma}(5/6)=\gamma\cdot T(5/6)=5/6.
$$
By virtue of the self-affine property \eqref{selfAffineProperty}, and using arguments similar to those in Lemma \ref{L1/2} and Proposition \ref{5/3}, we obtain the following result:

\begin{proposition}
For every $n\in \mathbb{N}$, if we let
$$
\gamma_n = \frac{2^{2n}+1}{2^{2n}-1}
$$
then $\textbf{T}_{\gamma_n}$ does not have the shadowing property.
\end{proposition}
\begin{proof}
Let $n\in\mathbb{N}$. Since $G'^+_{2(n-1)}(z_{n-1})=0$, it follows from \eqref{selfAffineProperty} that 
$$
T\left (z_{n-1}+\frac{5}{6}\cdot \frac{1}{2^{2(n-1)}}\right )=z_{n-1}+\frac{1}{2^{2(n-1)}}T(5/6)=z_{n}.
$$
We seek $\gamma_n$ such that 
$$
\gamma_n\cdot T\left (z_{n-1}+\frac{5}{6}\cdot \frac{1}{2^{2(n-1)}}\right )=z_{n-1}+\frac{5}{6}\cdot \frac{1}{2^{2(n-1)}}.
$$
Equivalently, 
$$
\gamma_n\cdot z_{n} =z_{n-1}+\frac{5}{6}\cdot \frac{1}{2^{2(n-1)}}.
$$
Recalling formula \eqref{z_n}, the result follows immediately by using that 
$$
z_{n-1}=z_{n}-\frac{1}{2}\cdot \frac{1}{2^{2(n-1)}}\quad  \text{ and } \quad \frac{1}{z_n}=\frac{3}{2}\cdot \frac{2^{2n}}{2^{2n}-1}.
$$
\end{proof}

\section{Final comments and questions}

The family of tent maps consists of the piecewise linear maps $f_s:[0,2]\to[0,2]$, $\sqrt{2}\leq s \leq 2$ defined by 
$$
f_s(x)=\begin{cases}
sx&\text{ if }0\leq x\leq 1,\\
s(2-x)& \text{ if }1\leq x\leq 2.\\
\end{cases}
$$
In \cite{CKY}, the authors investigated the shadowing property in the family of the tent maps. Among other results, they proved that the tent map $f_s$ has the shadowing property for almost every parameter $s$, while there also exists an uncountable dense set of parameters for which the corresponding tent maps fail to have the shadowing property.

In the context of the Takagi family, the main part of this work is devoted to establishing that the Takagi function ($\gamma=1$) has the shadowing property. Nonetheless, it is also shown in Section \ref{NoShadowingSection} that there are  infinitely many values of $\gamma>0$ for which $\textbf{T}_{\gamma}$ does not have the shadowing property. 

Following the spirit of the results obtained in \cite{CKY}, we pose the following open questions:

\begin{open}
Does the function $\textbf{T}_{\gamma}$ have the shadowing property for almost every parameter $\gamma>0$?
\end{open}
\begin{open}
Does there exist an uncountable dense set of parameters $\gamma>0$ for which the corresponding functions $\textbf{T}_{\gamma}$ fails to have the shadowing property?
\end{open}

\bibliographystyle{abbrv}
\bibliography{refs}

@article {A1,
    AUTHOR = {Allaart, Pieter C.},
     TITLE = {The finite cardinalities of level sets of the {T}akagi
              function},
   JOURNAL = {J. Math. Anal. Appl.},
  FJOURNAL = {Journal of Mathematical Analysis and Applications},
    VOLUME = {388},
      YEAR = {2012},
    NUMBER = {2},
     PAGES = {1117--1129},
      ISSN = {0022-247X,1096-0813},
   MRCLASS = {26A30 (26A18 26A27)},
  MRNUMBER = {2869811},
MRREVIEWER = {Tam\'as\ M\'atrai},
       DOI = {10.1016/j.jmaa.2011.10.060},
       URL = {https://doi.org/10.1016/j.jmaa.2011.10.060},
}

@article {A2,
    AUTHOR = {Allaart, Pieter C.},
     TITLE = {How large are the level sets of the {T}akagi function?},
   JOURNAL = {Monatsh. Math.},
  FJOURNAL = {Monatshefte f\"ur Mathematik},
    VOLUME = {167},
      YEAR = {2012},
    NUMBER = {3-4},
     PAGES = {311--331},
      ISSN = {0026-9255,1436-5081},
   MRCLASS = {26A27 (54E52)},
  MRNUMBER = {2961286},
MRREVIEWER = {Pratulananda\ Das},
       DOI = {10.1007/s00605-012-0390-0},
       URL = {https://doi.org/10.1007/s00605-012-0390-0},
}

@article {AK,
    AUTHOR = {Allaart, Pieter C. and Kawamura, Kiko},
     TITLE = {The {T}akagi function: a survey},
   JOURNAL = {Real Anal. Exchange},
  FJOURNAL = {Real Analysis Exchange},
    VOLUME = {37},
      YEAR = {2011/12},
    NUMBER = {1},
     PAGES = {1--54},
      ISSN = {0147-1937,1930-1219},
   MRCLASS = {26A27 (26A16 28A80)},
  MRNUMBER = {3016850},
MRREVIEWER = {Roland\ Girgensohn},
       URL = {http://projecteuclid.org/euclid.rae/1335806762},
}

@article {An,
    AUTHOR = {Anosov, D. V.},
     TITLE = {Geodesic flows on closed {R}iemannian manifolds of negative
              curvature},
   JOURNAL = {Trudy Mat. Inst. Steklov.},
  FJOURNAL = {Akademiya Nauk SSSR. Trudy Matematicheskogo Instituta imeni V.
              A. Steklova},
    VOLUME = {90},
      YEAR = {1967},
     PAGES = {209},
      ISSN = {0371-9685},
   MRCLASS = {57.36 (28.00)},
  MRNUMBER = {224110},
MRREVIEWER = {M.\ Eisenberg},
}

@article {Bi,
    AUTHOR = {Birkhoff, George D.},
     TITLE = {An extension of {P}oincar\'e's last geometric theorem},
   JOURNAL = {Acta Math.},
  FJOURNAL = {Acta Mathematica},
    VOLUME = {47},
      YEAR = {1926},
    NUMBER = {4},
     PAGES = {297--311},
      ISSN = {0001-5962,1871-2509},
   MRCLASS = {99-04},
  MRNUMBER = {1555218},
       DOI = {10.1007/BF02559515},
       URL = {https://doi.org/10.1007/BF02559515},
}

@article {Bo,
    AUTHOR = {Bowen, Rufus},
     TITLE = {{$\omega $}-limit sets for axiom {${\rm A}$} diffeomorphisms},
   JOURNAL = {J. Differential Equations},
  FJOURNAL = {Journal of Differential Equations},
    VOLUME = {18},
      YEAR = {1975},
    NUMBER = {2},
     PAGES = {333--339},
      ISSN = {0022-0396,1090-2732},
   MRCLASS = {58F10},
  MRNUMBER = {413181},
MRREVIEWER = {Zbigniew\ Nitecki},
       DOI = {10.1016/0022-0396(75)90065-0},
       URL = {https://doi.org/10.1016/0022-0396(75)90065-0},
}

@article {Bu,
    AUTHOR = {Buczolich, Z.},
     TITLE = {Irregular 1-sets on the graphs of continuous functions},
   JOURNAL = {Acta Math. Hungar.},
  FJOURNAL = {Acta Mathematica Hungarica},
    VOLUME = {121},
      YEAR = {2008},
    NUMBER = {4},
     PAGES = {371--393},
      ISSN = {0236-5294,1588-2632},
   MRCLASS = {28A75 (26A24 26A27)},
  MRNUMBER = {2461441},
MRREVIEWER = {Ville\ Suomala},
       DOI = {10.1007/s10474-008-7220-9},
       URL = {https://doi.org/10.1007/s10474-008-7220-9},
}

@article {CKY,
    AUTHOR = {Coven, Ethan M. and Kan, Ittai and Yorke, James A.},
     TITLE = {Pseudo-orbit shadowing in the family of tent maps},
   JOURNAL = {Trans. Amer. Math. Soc.},
  FJOURNAL = {Transactions of the American Mathematical Society},
    VOLUME = {308},
      YEAR = {1988},
    NUMBER = {1},
     PAGES = {227--241},
      ISSN = {0002-9947,1088-6850},
   MRCLASS = {58F30 (34C35)},
  MRNUMBER = {946440},
MRREVIEWER = {Georgi\u i\ Osipenko},
       DOI = {10.2307/2000960},
       URL = {https://doi.org/10.2307/2000960},
}

@article {FGGLl2,
    AUTHOR = {Ferrera, Juan and G\'omez Gil, Javier and Llorente, Jes\'us},
     TITLE = {Superdifferential analysis of the {T}akagi-{V}an der
              {W}aerden functions},
   JOURNAL = {Set-Valued Var. Anal.},
  FJOURNAL = {Set-Valued and Variational Analysis. Theory and Applications},
    VOLUME = {30},
      YEAR = {2022},
    NUMBER = {3},
     PAGES = {811--826},
      ISSN = {1877-0533,1877-0541},
   MRCLASS = {49J52 (26A27 26A30 28A78)},
  MRNUMBER = {4455147},
       DOI = {10.1007/s11228-021-00620-1},
       URL = {https://doi.org/10.1007/s11228-021-00620-1},
}

@article{HYG,
title = {Do numerical orbits of chaotic dynamical processes represent true orbits?},
journal = {Journal of Complexity},
volume = {3},
number = {2},
pages = {136-145},
year = {1987},
issn = {0885-064X},
doi = {https://doi.org/10.1016/0885-064X(87)90024-0},
url = {https://www.sciencedirect.com/science/article/pii/0885064X87900240},
author = {Stephen M Hammel and James A Yorke and Celso Grebogi}
}

@article {HLL,
    AUTHOR = {Hanson, Bruce and Llorente, Jes\'us},
     TITLE = {The little lip and big {L}ip functions of the {T}akagi
              function},
   JOURNAL = {Rev. R. Acad. Cienc. Exactas F\'is. Nat. Ser. A Mat. RACSAM},
  FJOURNAL = {Revista de la Real Academia de Ciencias Exactas, F\'isicas y
              Naturales. Serie A. Matematicas. RACSAM},
    VOLUME = {120},
      YEAR = {2026},
    NUMBER = {1},
     PAGES = {Paper No. 27, 18},
      ISSN = {1578-7303,1579-1505},
   MRCLASS = {26A27 (26A24 26A30)},
  MRNUMBER = {5004708},
       DOI = {10.1007/s13398-025-01815-z},
       URL = {https://doi.org/10.1007/s13398-025-01815-z},
}

@article {Ka,
    AUTHOR = {Kahane, Jean-Pierre},
     TITLE = {Sur l'exemple, donn\'e{} par {M}. de {R}ham, d'une fonction
              continue sans d\'eriv\'ee},
   JOURNAL = {Enseign. Math. (2)},
  FJOURNAL = {L'Enseignement Math\'ematique. Revue Internationale. 2e
              S\'erie},
    VOLUME = {5},
      YEAR = {1959},
     PAGES = {53--57},
      ISSN = {0013-8584},
   MRCLASS = {26.00},
  MRNUMBER = {108556},
MRREVIEWER = {U.\ S.\ Haslam-Jones},
}

@incollection {L,
    AUTHOR = {Lagarias, Jeffrey C.},
     TITLE = {The {T}akagi function and its properties},
 BOOKTITLE = {Functions in number theory and their probabilistic aspects},
    SERIES = {RIMS K\^oky\^uroku Bessatsu},
    VOLUME = {B34},
     PAGES = {153--189},
 PUBLISHER = {Res. Inst. Math. Sci. (RIMS), Kyoto},
      YEAR = {2012},
   MRCLASS = {26A27 (11A63)},
  MRNUMBER = {3014845},
MRREVIEWER = {Kiko\ Kawamura},
}

@inproceedings{Palmer,
  title={Shadowing in Dynamical Systems: Theory and Applications},
  author={Kenneth James Palmer},
  year={2010},
  url={https://api.semanticscholar.org/CorpusID:117878264}
}

@inproceedings{Pilyugin,
  title={Shadowing in dynamical systems},
  author={Sergei Yu. Pilyugin},
  year={1999},
  url={https://api.semanticscholar.org/CorpusID:118797227}
}

@article {T,
    AUTHOR = {Takagi, Teiji},
     TITLE = { A simple example of the continuous function without derivative},
   JOURNAL = {Proc. Phys. Math. Soc. Tokio Ser. II},
  FJOURNAL = {Proc. Phys. Math. Soc. Tokio Ser. II},
    VOLUME = {},
      YEAR = {1903},
    NUMBER = {1},
     PAGES = {176--177},
       DOI = {10.11429/subutsuhokoku1901.1.F176},
       URL = {https://doi.org/10.11429/subutsuhokoku1901.1.F1767},
}

@article {TK,
    AUTHOR = {Gedeon, Tom\'a\v s{} and Kuchta, Milan},
     TITLE = {Shadowing property of continuous maps},
   JOURNAL = {Proc. Amer. Math. Soc.},
  FJOURNAL = {Proceedings of the American Mathematical Society},
    VOLUME = {115},
      YEAR = {1992},
    NUMBER = {1},
     PAGES = {271--281},
      ISSN = {0002-9939,1088-6826},
   MRCLASS = {58F12 (58F20)},
  MRNUMBER = {1086325},
MRREVIEWER = {Remo\ Badii},
       DOI = {10.2307/2159597},
       URL = {https://doi.org/10.2307/2159597},
}

@article {YT,
    AUTHOR = {Yamaguchi, Yoshihiro and Tanikawa, Kiyotaka and Mishima,
              Nobuhiko},
     TITLE = {Fractal basin boundary in dynamical systems and the
              {W}eierstrass-{T}akagi functions},
   JOURNAL = {Phys. Lett. A},
  FJOURNAL = {Physics Letters. A},
    VOLUME = {128},
      YEAR = {1988},
    NUMBER = {9},
     PAGES = {470--478},
      ISSN = {0375-9601,1873-2429},
   MRCLASS = {58F08 (28A75 30D05)},
  MRNUMBER = {940616},
       DOI = {10.1016/0375-9601(88)90878-X},
       URL = {https://doi.org/10.1016/0375-9601(88)90878-X},
}

@article{JLL,
  title={Generalized Takagi functions},
  author={Llorente, Jesús},
  year={2024},
  url={https://hdl.handle.net/20.500.14352/105164},
journal={PhD thesis},
}

@misc{BLL,
  author       = {Buczolich, Zolt\'an and Llorente, Jes\'us},
  title        = {Differentiability properties of iterates of the Takagi family},
  year         = {2026},
  howpublished = {To be submitted},
  note         = {},
  url          = {}
}

\end{document}